\documentclass[10pt]{article}
\usepackage{geometry} 
\usepackage{amssymb,amsmath,amsthm}
\usepackage{times,color}
\usepackage{url,multirow}
\usepackage{graphicx}
\usepackage{epstopdf}
\newtheorem{theorem}{Theorem}[section]

\newtheorem{corollary}{Corollary}[section]
\newtheorem{remark}{Remark}[section]

\numberwithin{equation}{section}
\title{Stochastic velocity motions and processes with random time}
\author{ \Large Alessandro De Gregorio\\
 Department of Statistics, Probability and Applied Statistics\\
``Sapienza'' University of Rome\\\
 P.le Aldo Moro, 5 - 00185, Rome - Italy \\
 alessandro.degregorio@uniroma1.it}
\begin{document}
\maketitle
\begin{abstract}

The aim of this paper is to analyze a class of random motions which models the motion of a particle on the real line with random velocity and subject to the action of the friction.  The speed randomly changes when a Poissonian event occurs. We study the characteristic and the moment generating function of the position reached by the particle at time $t>0$. We are able to derive the explicit probability distributions in few cases for which discuss the connections with the random flights. The moments are also widely analyzed.  

For the random motions having an explicit density law, further interesting probabilistic interpretations emerge if we deal with them varying up a random time. Essentially, we consider two different type of random times, namely Bessel and Gamma times, which contain, as particular cases, some important probability distributions (e.g. Gaussian, Exponential). In particular, for the random processes built by means of these compositions, we derive the probability distributions fixed the number of Poisson events.

Some remarks on the possible extensions to the random motions in higher spaces are proposed. We focus our attention on the persistent planar random motion.
\\

 
\end{abstract}

\section{Introduction}
Diffusion processes have a central position in the theory of probability. Nevertheless, their main shortcoming is the unboundeness of the first variation. For this reason a diffusion process is often not suitable to describe the real motion and many researchers have proposed alternative models having finite speed.  

The prototype of the random motions with finite velocity is the telegraph process. By assuming that the change of direction is governed by a homogeneous Poisson process $N(t)$ with rate $\lambda>0$, we can define the telegraph process as follows
\begin{equation}\label{telegraph}
T(t)=V(0)\sum_{j=1}^{N(t)+1}(s_j-s_{j-1})(-1)^{j-1},
\end{equation}
where $V(0)$ is the initial velocity assuming the values $+c$ or $-c$ with probability $\frac12$ and the times  $s_j$ are the instants in which the $j$-th Poisson event occurs. Furthermore, $T(t)$ is linked with the hyperbolic partial differential equations, because its density law is the fundamental solution of the equation
\begin{equation}\label{pde}
\frac{\partial^2 u}{\partial t^2}+\lambda \frac{\partial u}{\partial t}=c^2\frac{\partial^2 u}{\partial x^2}.
\end{equation}
  
The telegraph process has been studied by several authors; for example Orsingher (1990), Foong and Kanno (1994), Di Crescenzo (2001), Stadje and Zacks (2004) and Zacks (2004). This model seems to be suitable to describe the real motion and it emerges in different fields. In the following, we provide a brief review of the possible applications of the telegraph process and its generalizations.
\begin{itemize}
\item {\bf Physics}. In the physical mathematics the connection between the telegraph process and the electromagnetic theory strongly emerges. In particular, the equation \eqref{pde} describes the propagation of a damped wave along a wire. Weiss (2002) provided an interesting review of the physical applications of the process $T(t)$.

\item {\bf Biology}. Models governed by hyperbolic differential equations and in particular the telegraph equation have been used to describe the movement of chemotaxis (see Hillen {\it et al.}, 2000).

\item{\bf Ecology}. The telegraph process
is useful in ecology to model
the displacement of wild animals on the soil (see Holmes {\it et al.}, 1994). In fact, this model
preserves the property of animals to move at finite velocity along
the same direction.

\item {\bf Finance}. Di Masi {\it et al.} (1994) proposed to model the volatility of financial markets in
terms of $T(t)$. Di Crescenzo and Pellerey (2002) introduced the
geometric telegraph process as a model to describe the dynamics of the price of risky assets, i.e. the authors replaced the standard Brownian motion with the standard telegraph process. Ratanov (2007, 2008) proposed to model financial markets
using the telegraph process with two different velocities (as the risky asset tends upward or downward) and jumps occurring at switching
velocities.

\item {\bf Actuarial Sciences}. Mazza and Rulliere (2004) established a link between hitting times associated with the risk process (time of ruin of the insurance company) and the telegrapher's motion.

\end{itemize}

A statistical analysis of the random model $T(t)$ has been performed by De Gregorio and Iacus (2008) and Iacus and Yoshida (2009), when the sample path is observed at discrete times.

In this paper we will analyze a one-dimensional random motion which in somehow generalized $T(t)$. At time $t>0$, the random speed of the motion is defined by $v=c\cos\theta$, where $c$ is a positive constant and $\theta$ is a random variable with density law given by
$$f_\nu(\theta)=\frac{\Gamma(\nu+1)}{\sqrt{\pi}\Gamma(\nu+\frac12)}\sin^{2\nu}\theta,\quad \theta\in(0,\pi), \,\nu \geq 0.$$
So, we consider a particle starting from the origin, choosing initially a velocity $c\cos\theta_1$ with probability law given by $f_\nu(\theta)$. The particle travels maintaining its motion with the same velocity until a Poisson event happens. Now, the particle changes velocity independently to the previous one according to $f_\nu(\theta)$ again and so on. At time $t$, we indicate the particle position on the real line with $X_\nu(t)$. 

The function $f_\nu(\theta)$ allows us to define a random motion in which the small displacements have bigger probability than large ones. Therefore, we have the physical phenomenon called {\it inertia}. Indeed, every object or particle moving on a surface suffers an effect due to the friction of the surface itself. Then, the particle will tend to go away from  the starting point slowly. If $\nu=0$, one has that $f_0(\theta)=\frac{1}{\pi}$ becomes  the uniform distribution on the semicircle with radius one. In this last case the particle moves without inertia.

In the same spirit of Orsingher and De Gregorio (2007b), we study the conditional characteristic and moment generating function of $X_\nu(t)$. We are able to derive the explicit probability distribution, conditioned to number of Poisson event, in few cases: $\nu=0$ and $\nu=1$. Therefore, let $N(t)$ denote the underlying Poisson process governing the changes of the velocity, we have that
\begin{align*}
&\frac{P(X_0(t)\in dx|N(t)=n)}{dx}=\frac{\Gamma\left(\frac n2\right) \Gamma\left(\frac n2+1\right)}{2\pi\Gamma(n)} 
\left(\frac{2}{ct}\right)^{n}(c^2t^2-x^2)^{\frac{n-1}{2}},\\
&\frac{P(X_1(t)\in dx|N(t)=n)}{dx}=\frac{\Gamma\left(n+1\right) \Gamma\left(n+2\right)}{2\pi\Gamma(2n+2)} 
\left(\frac{2}{ct}\right)^{2n+2}(c^2t^2-x^2)^{n+\frac{1}{2}},
\end{align*}
with $n\geq 1$.
These results, permit us to put in light the relationship between $X_0(t)$ and $X_1(t)$ and the random flights studied by Orsingher and De Gregorio (2007b). A random flight is a continuous time random walk defined similarly to $X_\nu(t)$, but with direction chosen uniformly on an hypersphere. By means of the above probabilities, we can claim that, in distribution, $X_0(t)$ and $X_1(t)$ correspond respectively to the projection onto real axis of a planar and four dimensional random flight.

In Section 3 we derive the first two moments of $X_\nu(t)$, while for $\nu=0,1,$ we are able to explicit the moments of order $p$ by means of special functions. Moreover, we will point out some connections with the related random motions on hyperbolic spaces.

In the second part of this paper, we focus our attention on the random motions $X_0(t)$ and $X_1(t)$ evolving up a random time, leading to interesting interpretations of the related conditional density laws. In other words, we will introduce families of random times, containing some important random variables.  In the probabilistic literature there are several papers devoted to analyze the properties of stochastic processes with random times. For example, the Brownian motion with Brownian time (iterated Brownian motion) has been studied by Burdzy (1993), Khoshnevisan and Lewis 
(1996), Allouba (2002), DeBlassie (2004) and Nane (2006). A link between the solution of fractional partial differential equations and the iterated Brownian motion has been extensively investigated by Orsingher and Beghin (2009). The iterated Brownian motion has been proposed as a model for a diffusion in a 
crack (see Burdzy and Khoshnevisan, 1998). Beghin and Orsingher (2009) have studied a planar random motion with Brownian times; the authors have provided the conditional probability on the number of the events of a fractional Poisson process. 

It is interesting to consider random times derived by Brownian motion. For example, the Bessel process $R^d(t)=\sqrt{\sum_{j=1}^d B_j(t)}$, where $B_j$ are independent Brownian motions. Under the condition $v_m>\frac d2-1$, $m=0,1,$ the following result holds
\begin{equation*}
P\left\{X_m(R^d(t))\in dx|N(t)=n\right\}=\frac{dx}{B\left(\frac d2,v_m-\frac d2+1\right)}\int_0^1w^{\frac d2-1}(1-w)^{v_m-\frac d2}\frac{e^{-\frac{x^2}{2c^2t w}}}{\sqrt{2\pi tw}c}dw
\end{equation*}
where $v_0=\frac n2$, $v_1=n+1$, $n\geq 1$ and $B(a,b)$ is a Beta function of parameters $a$ and $b$. Hence, the random motion $X_m$ stopped at  Bessel random time changes drastically its probability distribution which becomes a Gaussian with variance given by a Beta random variable (up the scale factor $c^2 t$). These results can be generalized by considering a $l$-times interated Bessel process $\mathcal{R}_l^d(t)=R_1^d(R_2^d(...(R_{l+1}^d(t)...)))$.

For $d=1$, the Bessel process becomes a reflected (around zero) Brownian motion $|B(t)|$, that is the Brownian time arising in the iterated Brownian motion. Then, the above condition is always satisfied and the following distribtional equality emerges
$$X_m(|B(t)|)\stackrel{d}{=}B\left(\frac1tX_m^2(t)\right).$$
For $d=2$, we obtain a time distributed as a Rayleigh random variable and $v_m$ is strictly positive for each $n\geq 1$. This last fact permits us to provide the density law of $X_m(R^2(t))$.

Other relationships will be point out considering the composition with the sojourn time of a Brownian motion on the positive axis.

 Let $G_\alpha(t)$ be
 a Gamma process with parameters $\alpha>0$, that is a random process with probability law given by a $g_\alpha(s,t)=\frac{t^\alpha}{\Gamma(\alpha)}t^{\alpha -1}e^{-ts}, s>0,t>0$. The Gamma random variable arising in various applications and it is useful to model the lifetime of a phenomena. For this reason, we deal with a Gamma time and for $X_m(G_\alpha(t))$ obtain that
  \begin{eqnarray*}
 &&P\left\{X_m(G_\alpha(t))\in dx|N(t)=n\right\}\\
 &&=\frac{dx}{\Gamma(\frac{\alpha+1}{2})B(\frac\alpha 2,v_m-\frac\alpha 2+1)} \int_0^1dw w^{\frac{\alpha}{2}-1}(1-w)^{v_m-\frac\alpha 2}\frac{t}{\sqrt{\pi w}c}\left(\frac{t|x|}{2c\sqrt{w}}\right)^{\frac\alpha 2}K_{-\frac\alpha 2}\left(\frac{t|x|}{c\sqrt{w}}\right)
 \end{eqnarray*}
with $v_m>\frac \alpha2-1$.
As for the Bessel process, this result can be extended considering the $l$-iterated Gamma random times $\mathcal{G}_l(t)=G_{\alpha_1}(G_{\alpha_2}(\ldots (G_{\alpha_{l+1}}(t))\ldots))$.
 
 For $\alpha=1$, $G_1(t)$ becomes an exponential process and  $X_m(G_1(t))$ is distributed as a Laplace random variable with parameter $\frac{t^2}{|X_m(t)|}$.

To complete the discussion on the random times, we deal with a clock obtained mixing $R^d(t)$ and $G_\alpha(t)$ and studying the effect on the probabilistic structure of the random motion $X_m(t)$.

It is not an hard task to extend the previous results to the planar and four dimensional random flights with randomly varying time. Moreover, in the last Section, we will discuss the possibility to consider a random flight with drift, that is persistent along a portion of the surface. 

\section{Moving randomly with friction}
Let us consider a random motion which describes the displacements of a particle starting from the origin of the real axis. The particle moves forward or backward with random velocity $v=c\cos\theta$, where $c$ is a positive constant, while $\theta$ is a random variable having density law 
\begin{equation}\label{dens:vel}
f_\nu(\theta)=\frac{\Gamma(\nu+1)}{\sqrt{\pi}\Gamma(\nu+\frac12)}\sin^{2\nu}\theta,
\end{equation}
with $\,\theta\in(0,\pi)$ and $\nu\geq 0$.
Therefore, the particle moves with a velocity, randomly chosen on $x$-component of the unit semicircle according to \eqref{dens:vel}, and it performs its motion until a Poisson event happens when another velocity will be chosen independently from the previous one. The position at time $t$ of the particle is defined as follows
\begin{equation}\label{gentel}
X_\nu(t)=c\sum_{j=1}^{N(t)+1}(s_j-s_{j-1})\cos\theta_j
\end{equation}
where $N(t)$ represents the underlying homogenous Poisson process with rate $\lambda>0$ governing the changes of velocity, $s_j,j=1,...,N(t)+1(s_0=0,s_{N(t)+1}=t)$, is the time of occurrence of the $j$-th Poisson event and $\theta_j$'s are independent random variables distributed as in \eqref{dens:vel}. Furthermore, also $N(t)$ and $\theta_j$ are independent.  From \eqref{gentel} emerges that $X_\nu(t)$ is a telegraph-type process similar to $T(t)$ defined in \eqref{telegraph}. We note that $X_\nu(t)$ has an infinite number of possible velocities and it has no necessarily alternating directions. Further, the particle at time $t$ is located inside the interval $(-ct,ct)$, and $X_\nu(t)$ has a fully absolutely continuous probability distribution, whilst in the law of $T(t)$ a singular component appears (see Orsingher, 1990).

\begin{figure}[t]
\centering{\includegraphics[width=10.5cm]{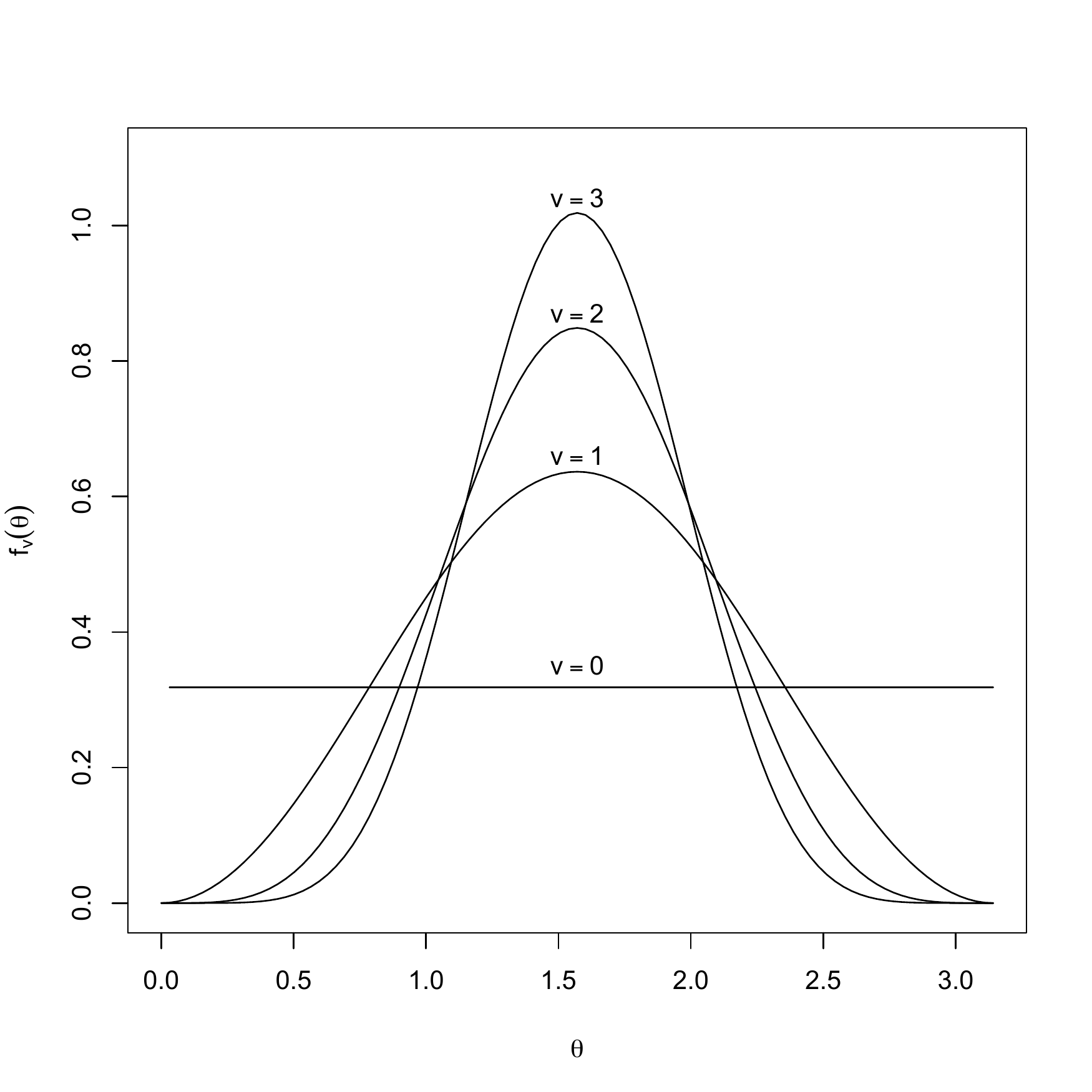}}
\caption{The behavior of $f_\nu(\theta)$ for $\nu=0,1,2,3$.}
\end{figure}

We underline that for values of $\theta$ close to $\frac\pi2$ the density law $f_\nu(\theta)$ assigns probability mass greater than ones near 0 or $\pi$. This means that the process moves away slowly from the starting point. This represents the effect of the friction of the surface on which the particle performs its motion.  When $\nu$ assumes high values, the density $f_\nu(\theta)$ is highly concentrated aroung $\frac\pi 2$ (see Figure 1) and then the motion is slowed down. For this reason $\nu$ represents the level of friction to which is subject the motion. In other words, $X_\nu(t)$ defines a whole class of random motions indexed by the parameter $\nu$, namely the level of inertia. For $\nu=0$, we reobtain the uniform distribution on the semicircle with radius one and $X_0(t)$ is exactly the $x$-component of the planar random flights studied in Orsingher and De Gregorio (2007b) or equivalently the projection onto real line of the sample path of a planar random flight.

Our first result concerns the characteristic function of $X_\nu(t)$ conditioned on the number of Poisson events during the time interval $[0,t]$.
\begin{theorem}\label{chf}
The conditional charactersitic function of $X_\nu(t)$ is equal to
\begin{equation}\label{chfn}
E\left\{e^{i\alpha X_\nu(t)}|N(t)=n\right\}=\frac{n!}{t^n}(2^\nu\Gamma(\nu+1))^{n+1}\int_0^tds_1\cdots\int_{s_{n-1}}^tds_{n}\prod_{j=1}^{n+1}\frac{J_\nu(\alpha c(s_j-s_{j-1}))}{(\alpha c(s_j-s_{j-1}))^\nu},
\end{equation}
for $n\geq 1$, while for $n=0$, one has
\begin{equation}\label{chf0}
E\left\{e^{i\alpha X_\nu(t)}|N(t)=0\right\}=\left(\frac{2}{\alpha ct}\right)^\nu\Gamma(\nu+1) J_\nu(\alpha ct),
\end{equation}
where $J_\nu(x)=\sum_{k=0}^\infty(-1)^k\frac{(x/2)^{2k+\nu}}{\Gamma(k+1)\Gamma(k+\nu+1)}$ is the well-known Bessel function.
\end{theorem}
\begin{proof} In order to prove \eqref{chfn} and \eqref{chf0}, we observe that the Bessel function $J_\nu(x)$ admits the following integral representation
\begin{eqnarray}\label{irbess}
J_\nu(z)&=&\frac{\left(\frac z2\right)^\nu}{\Gamma\left(\nu+\frac12\right)\Gamma\left(\frac12\right)}\int_0^\pi e^{iz\cos\phi}\sin^{2\nu}\phi d\phi
\end{eqnarray}
with $Re(\nu)>0$. For $n\geq 1$,we get that
\begin{eqnarray*}
&&E\left\{e^{i\alpha X_\nu(t)}|N(t)=n\right\}\\
&&=\frac{n!}{t^n}\int_0^tds_1\cdots\int_{s_{n-1}}^tds_{n}\frac{\Gamma(\nu+1)}{\sqrt{\pi}\Gamma(\nu+\frac12)}\int_0^{\pi}\sin^{2\nu}\theta_1d\theta_1\cdots\frac{\Gamma(\nu+1)}{\sqrt{\pi}\Gamma(\nu+\frac12)}\int_0^{\pi}\sin^{2\nu}\theta_{n+1}d\theta_{n+1}\\
&&\quad\times\exp\left\{i\alpha c\sum_{j=1}^{n+1}(s_j-s_{j-1})\cos\theta_j\right\}\\
&&=\frac{n!}{t^n}\int_0^tds_1\cdots\int_{s_{n-1}}^tds_{n}\prod_{j=1}^{n+1}\left\{\frac{\Gamma(\nu+1)}{\sqrt{\pi}\Gamma(\nu+\frac12)}\int_0^{\pi}e^{i\alpha c(s_j-s_{j-1})\cos\theta_j}\sin^{2\nu}\theta_jd\theta_j\right\}\\
&&=\frac{n!}{t^n}(2^\nu\Gamma(\nu+1))^{n+1}\int_0^tds_1\cdots\int_{s_{n-1}}^tds_{n}\prod_{j=1}^{n+1}\frac{J_\nu(\alpha c(s_j-s_{j-1}))}{(\alpha c(s_j-s_{j-1}))^\nu}
\end{eqnarray*}
where in the last step we have used the integral representation \eqref{irbess}. For $N(t)=0$, the position of the particle at time $t$ is $X(t)=ct\cos\theta$ and then
\begin{eqnarray*}
E\left\{e^{i\alpha X_\nu(t)}|N(t)=0\right\}&=&\frac{\Gamma(\nu+1)}{\sqrt{\pi}\Gamma(\nu+\frac12)}\int_0^{\pi} e^{i\alpha ct\cos\theta}\sin^{2\nu}\theta d\theta=\left(\frac{2}{\alpha ct}\right)^\nu\Gamma(\nu+1) J_\nu(\alpha ct).
\end{eqnarray*}
\end{proof}
It is interesting to observe that the charecterstic function of $X_\nu(t)$ has the same structure of the one emerging in the problem of $d$-dimensional random flights (see formula (2.3) in Orsingher and De Gregorio, 2007b), where the parameter $\nu$ is replaced by $\frac d2-1$.

For the moment generating function we present the following Theorem.

\begin{theorem}\label{mg} 
The conditional moment generating function of $X_\nu(t)$ becomes
\begin{eqnarray}\label{genmom}
E\left\{e^{\beta X_\nu(t)}|N(t)=n\right\}=\frac{n!}{t^n}(2^\nu\Gamma(\nu+1))^{n+1}\int_0^tds_1\cdots\int_{s_{n-1}}^tds_{n}\prod_{j=1}^{n+1}\frac{I_\nu(\beta c(s_j-s_{j-1}))}{(\beta c(s_j-s_{j-1}))^\nu}\end{eqnarray}
given $n\geq 1$, while if $n=0$ the following expression yields
\begin{equation}\label{genmom0}
E\left\{e^{\beta X_\nu(t)}|N(t)=0\right\}=\left(\frac{2}{\beta ct}\right)^\nu\Gamma(\nu+1) I_\nu(\beta ct)
\end{equation}
where $I_\nu(x)=\sum_{k=0}^\infty\frac{(x/2)^{2k+\nu}}{\Gamma(k+1)\Gamma(k+\nu+1)}$ represents the modified Bessel function.
\end{theorem}
\begin{proof}
The proof for \eqref{genmom0} and \eqref{genmom} follows analogously to the one developed for \eqref{chfn} and \eqref{chf0}, noticing that
$$I_\nu(z)=\frac{\left(\frac z2\right)^\nu}{\Gamma\left(\nu+\frac12\right)\Gamma\left(\frac12\right)}\int_0^\pi e^{z\cos\phi}\sin^{2\nu}\phi d\phi$$
\end{proof}
The random motions obtained by setting $\nu=0$ and $\nu=1$ in the density law $f_\nu(\theta)$ have a special role in this paper. Indeed, for $X_0(t)$ and $X_1(t)$, we are able to explicit in closed form their characteristic and moment generating functions and successively the density laws.  In order to distinguish these important particular cases from the general random model $X_\nu(t)$, $\nu\geq 0$, we will indicate them in the rest of paper with $X_m(t),m=0,1$. Moreover, we will use the following notation: $v_0=\frac n2$ and $v_1=n+1$. 
\begin{corollary}\label{corcf}
For $X_m(t),\,m=0,1$, and $n\geq 1$, we have that
\begin{align}
&E\left\{e^{i\alpha X_m(t)}|N(t)=n\right\}=\frac{\Gamma\left(v_m+1\right)2^{v_m}}{(\alpha ct)^{v_m}}J_{v_m}(\alpha ct),\label{chfcondp}\\
&E\left\{e^{\beta X_m(t)}|N(t)=n\right\}=\frac{\Gamma\left(v_m+1\right)2^{v_m}}{(\beta ct)^{v_m}}I_{v_m}(\beta  ct).\label{genmom2}
\end{align}
\end{corollary}
\begin{proof}
We only give some sketches of the proof. Starting from \eqref{chfn}, we prove \eqref{chfcondp} for $m=0$.  It is possible to use the same approach used by Orsingher and De Gregorio (2007b), noticing that the $n$-fold integral 
$$\int_0^tds_1\cdots\int_{s_{n-1}}^tds_{n}\prod_{j=1}^{n+1}J_0(\alpha c(s_j-s_{j-1}))$$
can be worked out by applying recursively the following formula (see Gradshteyn and Ryzhik, 1980, formula 6.533(2))
\begin{equation}\label{formula}
\int_0^ax^{\mu}(a-x)^{\nu}J_\mu(x)J_\nu(a-x)dx=\frac{\Gamma(\mu+\frac12)\Gamma(\nu+\frac12)}{\sqrt{2\pi}\Gamma(\mu+\nu+1)}a^{\mu+\nu+\frac12}J_{\mu+\nu+\frac12}(a)
\end{equation}
with $Re(\mu)>-\frac12,\,Re(\nu)>-\frac12$.
Analogously, for $m=1$, by taking into account the formula (see Gradshteyn and Ryzhik, 1980, formula 6.581(3))
\begin{equation}\label{formula2}
\int_0^a\frac{J_\mu(x)}{x}\frac{J_\nu(a-x)}{a-x}dx=\left(\frac1\mu+\frac1\nu\right)\frac{J_{\mu+\nu}(a)}{a},
\end{equation}
with $Re(\mu)>0,\,Re(\nu)>0$,
it is possible to compute the exact value of the following quantity 
$$\int_0^tds_1\cdots\int_{s_{n-1}}^tds_{n}\prod_{j=1}^{n+1}\frac{J_1(\alpha c(s_j-s_{j-1}))}{\alpha c(s_j-s_{j-1})}.$$
Hence, the result \eqref{chfcondp} for $v_1=n+1$ emerges.

For the moment generating function, we need to prove that the following semigroup-type property holds
\begin{equation}\label{sem}
\int_0^ax^{\mu}(a-x)^{\nu}I_\mu(x)I_\nu(a-x)dx=\frac{\Gamma(\mu+\frac12)\Gamma(\nu+\frac12)}{\sqrt{2\pi}\Gamma(\mu+\nu+1)}a^{\mu+\nu+\frac12}I_{\mu+\nu+\frac12}(a)
\end{equation}
Indeed, since (see Gradshteyn and Ryzhik, 1980)
$$\int_0^\infty e^{\beta  x}x^\mu I_\mu(x)dx=\frac{2^\mu\Gamma(\mu+\frac12)}{\sqrt{\pi}(\beta^2-1)^{\mu+\frac12}},\quad Re(\mu)>-\frac12, $$
we have that
\begin{eqnarray*}
\int_0^\infty e^{\beta  a}da \int_0^ax^{\mu}(a-x)^{\nu}I_\mu(x)I_\nu(a-x)dx&=&\int_0^\infty e^{\beta  x}x^\mu I_\mu(x)dx\int_0^\infty e^{\beta  y}y^\nu I_\nu(y)dy\\
&=&2^{\mu+\nu}\pi^{-1}\frac{\Gamma\left(\mu+\frac12\right)\Gamma\left(\nu+\frac12\right)}{(\beta ^2-1)^{\mu+\nu+1}}\\
&=&\frac{\Gamma(\mu+\frac12)\Gamma(\nu+\frac12)}{\sqrt{2\pi}\Gamma(\mu+\nu+1)}\int_0^\infty e^{\beta  a}a^{\mu+\nu+\frac12}I_{\mu+\nu+\frac12}(a)da
\end{eqnarray*}
Furthermore, being (see Gradshteyn and Ryzhik, 1980)
$$\int_0^\infty e^{\beta  x}x^{-1} I_\mu(x)dx=\frac{\mu}{[\beta+(\beta^2-1)]^{\mu}},\quad Re(\mu)>0$$
with the similar steps used to obtain \eqref{sem}, it is possible to show that
\begin{equation}\label{sem2}
\int_0^a\frac{I_\mu(x)}{x}\frac{I_\nu(a-x)}{a-x}dx=\left(\frac{1}{\mu}+\frac{1}{\nu}\right)a^{-1}I_{\mu+\nu}(a)
\end{equation}
In conclusion, by means of \eqref{sem} and \eqref{sem2}, and the same considerations done for the proof of \eqref{chfcondp}, the proof \eqref{genmom2} of immediately follows.
\end{proof}
\begin{remark}\label{rem}
We are able to give an integral  representation of the unconditional characteristic function of $X_0(t)$. Indeed, we have that
\begin{eqnarray*}
E\left\{e^{i\alpha X_0(t)}\right\}&=&e^{-\lambda t}\sum_{n=0}^\infty\frac{(\lambda t)^n}{n!}\frac{\Gamma\left(\frac n2+1\right)2^{\frac n2}}{(\alpha ct)^{\frac n2}}J_{\frac n2}(\alpha ct)\\
&=&e^{-\lambda t}\sum_{n=0}^\infty\frac{(\lambda t)^n}{n!}\frac{\Gamma\left(\frac n2+1\right)}{\sqrt{\pi}\Gamma(\frac{n+1}{2})}\int_0^\pi e^{i\alpha ct\cos\theta}\sin^n\theta d\theta\\
&=&e^{-\lambda t}\sum_{n=0}^\infty\frac{(\lambda t)^n}{2^n\Gamma^2\left(\frac{n+1}{2}\right)}\int_0^\pi e^{i\alpha ct\cos\theta}\sin^n\theta d\theta\\
&=&e^{-\lambda t}\Bigg\{\sum_{m=0}^\infty\frac{(\lambda t)^{2m+1}}{2^{2m+1}\Gamma^2\left(m+1\right)}\int_0^\pi e^{i\alpha ct\cos\theta}\sin^{2m+1}\theta d\theta\\
&&+J_0(\alpha ct)+\sum_{m=0}^\infty\frac{(\lambda t)^{2m+2}}{2^{2m+2}\Gamma^2\left(m+\frac32\right)}\int_0^\pi e^{i\alpha ct\cos\theta}\sin^{2m+2}\theta d\theta\Bigg\}\\
&=&e^{-\lambda t}\left\{J_0(\alpha ct)+\frac{\lambda t}{2}\int_0^\pi e^{i\alpha ct\cos\theta}\left(I_0(\lambda t\sin\theta)+{\bf L}_0(\lambda t\sin\theta)\right)\sin\theta d\theta\right\}
\end{eqnarray*}
where ${\bf L}_\mu(x)=\sum_{k=0}^\infty\frac{(x/2)^{2k+\mu+1}}{\Gamma(k+\frac32)\Gamma(k+\mu+\frac32)}$ is the modified Struve function.

For $X_1(t)$ we get that
\begin{eqnarray*}
E\left\{e^{i\alpha X_1(t)}\right\}&=&e^{-\lambda t}\sum_{n=0}^\infty\frac{(\lambda t)^n}{n!}\frac{\Gamma\left(n+2\right)2^{n+1}}{(\alpha ct)^{n+1}}J_{n+1}(\alpha ct)\\
&=&e^{-\lambda t}\sum_{n=0}^\infty\frac{(\lambda t)^n}{\sqrt{\pi}}\frac{(n+\frac12-\frac12)+1}{\Gamma(n+\frac32)}\int_0^{\pi}e^{i\alpha ct\cos\theta}\sin^{2(n+1)}\theta d\theta\\
&=&\frac{e^{-\lambda t}}{\sqrt{\pi}}\int_0^{\pi}e^{i\alpha ct\cos\theta}\left\{E_{1,\frac12}(\lambda t\sin^2\theta)+\frac12E_{1,\frac32}(\lambda t\sin^2\theta)\right\} \sin^{2}\theta d\theta
\end{eqnarray*}
where $E_{\alpha,\beta}(x)=\sum_{k=0}^\infty\frac{x^k}{\Gamma(\alpha k+\beta)}$ is the Mittag-Leffler functon.
\end{remark}

As stated before $X_0(t)$ and $X_1(t)$ represent two important particular cases of the class of random motions $X_\nu(t)$. This is due to the fact that by means of \eqref{chfcondp} we derive their probability distributions.

\begin{theorem}\label{condlaw}
The following conditional density laws hold
\begin{align}\label{condp2}
&p_0^\nu(x,t)=\frac{P(X_\nu(t)\in dx|N(t)=0)}{dx}=\frac{\Gamma\left(\nu\right) \Gamma\left(\nu+1\right)}{2\pi\Gamma(2\nu)} 
\left(\frac{2}{ct}\right)^{2\nu}(c^2t^2-x^2)^{\nu-\frac{1}{2}}\\
\label{condp}
&p_n^m(x,t)=\frac{P(X_m(t)\in dx|N(t)=n)}{dx}=\frac{\Gamma\left(v_m\right) \Gamma\left(v_m+1\right)}{2\pi\Gamma(2v_m)} 
\left(\frac{2}{ct}\right)^{2v_m}(c^2t^2-x^2)^{v_m-\frac{1}{2}}
\end{align} 
with $|x|<ct$ and $m=0,1$.
\end{theorem}
\begin{proof} We only prove \eqref{condp} for $m=0$, because the other results follow by means of similar steps. Instead of inverting \eqref{chfcondp}, we show that the characteristic function of the probability distribution $p_n^0(x,t)$ corresponds to \eqref{chfcondp} for $m=0$. Therefore, we have that 
\begin{eqnarray*}
E\left\{e^{i\alpha X_0(t)}|N(t)=n\right\}&=&\frac{\Gamma\left(\frac n2\right) \Gamma\left(\frac n2+1\right)}{2\pi\Gamma(n)} 
\left(\frac{2}{ct}\right)^n \int_{-ct}^{ct} e^{i\alpha x}(c^2t^2-x^2)^{\frac{n-1}{2}}dx\\
&=&\frac{\Gamma\left(\frac n2\right) \Gamma\left(\frac n2+1\right)}{\pi\Gamma(n)} 
\left(\frac{2}{ct}\right)^n \int_0^{ct} \cos(\alpha x)(c^2t^2-x^2)^{\frac{n-1}{2}}dx\\
&=&\frac{\Gamma\left(\frac n2\right) \Gamma\left(\frac n2+1\right)}{\pi\Gamma(n)} 
\left(\frac{2}{ct}\right)^n\sum_{k=0}^\infty (-1)^k\frac{\alpha^{2k}}{(2k)!} \int_0^{ct} x^{2k}(c^2t^2-x^2)^{\frac{n-1}{2}}dx\\
&=&\frac{\Gamma\left(\frac n2\right) \Gamma\left(\frac n2+1\right)2^{n-1}}{\pi\Gamma(n)} \sum_{k=0}^\infty (-1)^k\frac{(\alpha ct)^{2k}}{(2k)!} \int_0^1y^{k-\frac12}(1-y)^{\frac{n-1}{2}}dy\\
&=&\frac{\Gamma\left(\frac n2\right) \Gamma\left(\frac n2+1\right)\Gamma\left(\frac {n+1}{2}\right)2^{n-1}}{\pi\Gamma(n)} \sum_{k=0}^\infty (-1)^k\frac{(\alpha ct)^{2k}}{(2k)!}\frac{\Gamma\left(k+\frac12\right)}{\Gamma\left(k+\frac {n}{2}+1\right)}\\
&=&\frac{\Gamma\left(\frac n2\right) \Gamma\left(\frac n2+1\right)\Gamma\left(n\right)\sqrt{\pi}}{\pi\Gamma(n) \Gamma\left(\frac n2\right)}\sum_{k=0}^\infty (-1)^k\frac{(\alpha ct)^{2k}}{(2k)!}\frac{\Gamma\left(2k\right)\sqrt{\pi}2^{1-2k}}{\Gamma\left(k\right)\Gamma\left(k+\frac {n}{2}+1\right)}\\
&=&\Gamma\left(\frac n2+1\right)\sum_{k=0}^\infty (-1)^k\left(\frac{\alpha ct}{2}\right)^{2k}\frac{1}{\Gamma\left(k+1\right)\Gamma\left(k+\frac {n}{2}+1\right)}\\
&=&\frac{\Gamma\left(\frac n2+1\right)2^{\frac n2}}{(\alpha ct)^{\frac n2}}J_{\frac n2}(\alpha ct).
\end{eqnarray*}
\end{proof}
\begin{remark}\label{rem}
Theorem \ref{condlaw} permits us to point out the existing connection between $X_0(t)$, $X_1(t)$ and the random flights. Indeed, by setting $\nu=0$ in $f_\nu(\theta)$, we reobtain the uniform distribution on a unit semicircle and (as expected) $X_0(t)$ represents the projection onto the $x$-axis of a planar random flight. Then, $$p_n^0(x,t)=\frac{\Gamma\left(\frac n2\right) \Gamma\left(\frac n2+1\right)}{2\pi\Gamma(n)} 
\left(\frac{2}{ct}\right)^{n}(c^2t^2-x^2)^{\frac{n-1}{2}}$$ corresponds to the distribution (4.4) obtained Orsingher and De Gregorio (2007b).  For $\nu=1$, the probability law 
$$p_n^1(x,t)=\frac{\Gamma\left(n+1\right) \Gamma\left(n+2\right)}{2\pi\Gamma(2n+2)} 
\left(\frac{2}{ct}\right)^{2n+2}(c^2t^2-x^2)^{n+\frac{1}{2}}$$
is the same of the one obtained by means of the projection of a random flight in $\mathbb{R}^4$ onto the real line (see (4.1c) in Orsingher and De Gregorio, 2007b). In other words, the shadow on $\mathbb{R}$ of a four-dimensional random flight is perceived by an observer located on the real line, as a slowed down motion. Therefore, in distribution, we have the following equality
$$X_1(t)\stackrel{d}{=}c\sum_{j=1}^{N(t)+1}(s_j-s_{j-1})\sin\theta_{1,j}\sin\theta_{2,j}\sin\phi$$
where $(\theta_{1,j}\theta_{2,j},\phi)$ is uniformly distributed on the four-dimensional hypersphere.
\end{remark}
\begin{remark}
Theorem \ref{condlaw} says us that $p_0^\nu(x,t)=p_{n}^m(x,t)$ if and only if $\nu=v_m$. This means that$X_0(t)$ (absence of inertia) and $X_1(t)$ is equivalent in distribution to a random model representing a particle slowly moving with the same speed  until $t$.  
\end{remark}

The unconditional density laws of $X_0(t)$ and $X_1(t)$ are given by the following expressions\begin{eqnarray}\label{law}
p^0(x,t)&=&\frac{\lambda e^{-\lambda t}}{2c}\sum_{k=0}^{\infty}\left(\frac{\lambda}{2c}\sqrt{c^2t^2-x^2}\right)^{k-1}\frac{1}{\Gamma^2(\frac{k+1}{2})}\\
&=&\frac{\lambda e^{-\lambda t}}{2c}\left[I_0\left(\frac\lambda c\sqrt{c^2t^2-x^2}\right)+{\bf L}_0\left(\frac\lambda c\sqrt{c^2t^2-x^2}\right)\right]+\frac{e^{-\lambda t}}{\pi\sqrt{c^2t^2-x^2}}\notag
\end{eqnarray} 
\begin{eqnarray}\label{law2}
p^1(x;t)&=&\frac{e^{-\lambda t}}{c\sqrt{\lambda \pi t^3}}\sum_{k=0}^{\infty}\left(\frac{\lambda}{c^2 t}(c^2t^2-x^2)\right)^{k+\frac12}\frac{k+1}{\Gamma(k+\frac{3}{2})}\\
&=&\frac{e^{-\lambda t}\sqrt{c^2t^2-x^2}}{(ct)^2\sqrt{\pi}}\left\{E_{1,\frac12}\left(\frac{\lambda}{c^2t}(c^2t^2-x^2)\right)+\frac12E_{1,\frac32}\left(\frac{\lambda}{c^2t}(c^2t^2-x^2)\right)\right\}\notag
\end{eqnarray}
with $|x|< ct$.

We point out that the first term in \eqref{law}, that is $I_0\left(\frac\lambda c\sqrt{c^2t^2-x^2}\right)$, is equal to the one presents in the absolutely continuous component of the law of a telegraph process (see Orsingher, 1990), that is
$$\frac{e^{-\lambda t}}{2c}\left[\lambda I_0\left(\frac\lambda c\sqrt{c^2t^2-x^2}\right)+\frac{\partial}{\partial t}I_0\left(\frac\lambda c\sqrt{c^2t^2-x^2}\right)\right],$$
while the derivative with respect to the time of the Bessel function is replaced by the modified Struve function ${\bf L}_0\left(\frac\lambda c\sqrt{c^2t^2-x^2}\right)$ (up the constant $\lambda$).

\section{On the moments and some relationships with random motions on hyperbolic spaces}
In this section we analyze the moments of the random motion $X_\nu(t)$. In particular,
we are able to provide the first two moments of $X_\nu(t)$ by applying the results contained in Section 3 of Stadje and Zacks (2004). Therefore, fixed $n\geq 0$, for the mean value one has
\begin{eqnarray*}
E\left\{X_\nu(t)|N(t)=n\right\}&=&ct E\left\{\cos\theta\right\}\\
&=&\frac{ct\Gamma(\nu+1)}{\sqrt{\pi}\Gamma(\nu+\frac12)}\int_{0}^\pi \cos\theta \sin^{2\nu}\theta d\theta\\
&=&\frac{ct\Gamma(\nu+1)}{\sqrt{\pi}\Gamma(\nu+\frac12)}\left\{\int_{0}^{\frac\pi2} \cos\theta \sin^{2\nu}\theta d\theta-\int_{0}^{\frac\pi2}\sin\theta\cos^{2\nu}\theta d\theta\right\}\\
&=&0
\end{eqnarray*}
where in the last step we have used the well-known integral
 \begin{equation}\label{wallis}
\int_{0}^{\frac\pi2}\sin^a\theta\cos^{b}\theta d\theta=\frac12\frac{\Gamma(\frac{a+1}{2})\Gamma(\frac{b+1}{2})}{\Gamma(\frac{a+b}{2}+1)},\quad Re(a)>-1,\,Re(b)>-1.
\end{equation}

\begin{remark}
The mean value can also be derived from Theorem \ref{chf} (or equivalently from Theorem \ref{mg}). It is clear that
$$\frac{J_\nu(\alpha c(s_j-s_{j-1}))}{(\alpha c(s_j-s_{j-1}))^\nu}\Big|_{\alpha=0}=\frac{1}{2^\nu\Gamma(\nu+1)}$$
while it is not hard to prove that
$$\frac{d}{d\alpha}\frac{J_\nu(\alpha c(s_j-s_{j-1}))}{(\alpha c(s_j-s_{j-1}))^\nu}=\sum_{k=0}^\infty \frac{k}{k!\Gamma(k+\nu+1)}\left(\frac{\alpha c(s_j-s_{j-1})}{2}\right)^{2k+\nu-1}\frac{c(s_j-s_{j-1})}{(\alpha c(s_j-s_{j-1}))^{\nu}}$$
which calculated at $\alpha=0$ is equal to 0.
Then
\begin{equation}\label{mom1}
E\left\{X_\nu(t)|N(t)=n\right\}=i^{-1}\frac{d}{d\alpha}E\left\{e^{i\alpha X_\nu(t)}|N(t)=n\right\}\Big|_{\alpha=0}=0
\end{equation}

\end{remark}
For the the second moment we have that
\begin{eqnarray*}
E\left\{X_\nu^2(t)|N(t)=n\right\}&=&\frac{2}{n+2}c^2t^2E\{\cos^2\theta\}\\
&=&\frac{2}{n+2}c^2t^2\frac{\Gamma(\nu+1)}{\sqrt{\pi}\Gamma(\nu+\frac12)}\int_0^\pi\cos^2\theta\sin^{2\nu}\theta d\theta\\
&=&\frac{2}{n+2}c^2t^2\frac{\Gamma(\nu+1)\Gamma(\frac32)}{\sqrt{\pi}\Gamma(\nu+2)}\\
&=&\frac{c^2t^2}{(\nu+1)(n+2)}
\end{eqnarray*}
and after some calculations
\begin{eqnarray}\label{mom2}
E\left\{X_\nu^2(t)\right\}&=&e^{-\lambda t}\frac{(ct)^2}{(\nu+1)}\sum_{n=0}^\infty\frac{(\lambda t)^n}{n!(n+2)}\notag\\
&=&e^{-\lambda t}\frac{(ct)^2}{(\nu+1)}\left\{E_{1,2}(\lambda t)-E_{1,3}(\lambda t)\right\}\notag\\
&=&\frac{(ct)^2}{(\nu+1)}\frac{1}{\lambda t}\left(1-\frac{1-e^{-\lambda t}}{\lambda t}\right)
\end{eqnarray}
where in the last step we have used the following relationships: $E_{1,2}(x)=\frac{e^x-1}{x}$ and $E_{1,3}=\frac{e^x-1-x}{x^2}$. As expected, if $\nu$ increases the action of the friction is stronger and the value of $E\left\{X_\nu^2(t)\right\}$ tends to decrease. Indeed, for growing values of $\nu$ the particle maintains itself close the starting point, so that the probability distribution of $X_\nu(t)$ will be less sparse.

Further, it is not difficult  to show that
\begin{equation}
E\{X_\nu^p(t)|N(t)=0\}=\int_{-ct}^{ct}x^pp_0^\nu(x,t)dx=\begin{cases}
     0 & \text{$p$ is odd }, \\
     \frac{\Gamma(\nu+1)\Gamma(\frac{p+1}{2})}{\sqrt{\pi}\Gamma(\frac p2+\nu+1)}(ct)^p& \text{$p$ is even}.
\end{cases}
\end{equation}
For $\nu=0,1$, we present the following result. 
\begin{theorem}\label{momp}
The $p$-th moment of $X_0(t)$ and $X_1(t)$ are respectively given by
\begin{align}
&E\{X_0^p(t)\}=e^{-\lambda t}\left(\frac{2}{\lambda t}\right)^{\frac{p-1}{2}}(ct)^{p}\Gamma\left(\frac{p+1}{2}\right)\left\{I_{\frac{p+1}{2}}(\lambda t)+{\bf L}_{\frac{p+1}{2}}(\lambda t)\right\}+\frac{e^{-\lambda t}(ct)^p\Gamma\left(\frac{p+1}{2}\right)}{\sqrt{\pi}\Gamma(\frac{p}{2}+1)}\label{mom0}\\
&E\{X_1^p(t)\}=\frac{e^{-\lambda t}}{\sqrt{\pi}}\Gamma\left(\frac{p+1}{2}\right)(ct)^p\left\{E_{1,\frac p2+1}(\lambda t)-\frac p2E_{1,\frac p2+2}(\lambda t)\right\}\label{mom1}
\end{align}
for $p$ even, whilst $E\{X_0^p(t)\}=E\{X_1^p(t)\}=0$ if $p$ is odd.
\end{theorem}

\begin{proof} Let $p$ be even, we get that
\begin{eqnarray}
E\{X_0^p(t)\}&=&\frac{\lambda e^{-\lambda t}}{c}\sum_{k=0}^\infty\left(\frac{\lambda}{2c}\right)^{k-1}\frac{1}{\Gamma^2(\frac{k+1}{2})}\int_0^{ct}x^p(c^2t^2-x^2)^{\frac{k-1}{2}}dx\notag\\
&=&(x=ct\sqrt{y})\notag\\
&=&\frac{\lambda e^{-\lambda t}}{2c}\sum_{k=0}^\infty\left(\frac{\lambda}{2c}\right)^{k-1}\frac{(ct)^{p+k}}{\Gamma^2(\frac{k+1}{2})}\int_0^1y^{\frac{p+1}{2}-1}(1-y)^{\frac{k+1}{2}-1}dy\notag\\
&=&\frac{\lambda e^{-\lambda t}}{2c}(ct)^p\Gamma\left(\frac{p+1}{2}\right)\sum_{k=0}^\infty\left(\frac{\lambda}{2c}\right)^{k-1}\frac{(ct)^k}{\Gamma(\frac{k+1}{2})\Gamma(\frac{k+p}{2}+1)}\label{intstep}
\end{eqnarray}
Now, we splitting the above sum in order to carry out separately the even  and the odd elements. 
\begin{eqnarray*}
E\{X_0^p(t)\}&=&\frac{\lambda e^{-\lambda t}}{2c}(ct)^p\Gamma\left(\frac{p+1}{2}\right)\Bigg\{\sum_{k=0}^\infty\left(\frac{\lambda}{2c}\right)^{2k}\frac{(ct)^{2k+1}}{\Gamma(k+1)\Gamma(k+\frac{p+1}{2}+1)}\\
&&+\sum_{k=0}^\infty\left(\frac{\lambda}{2c}\right)^{2k-1}\frac{(ct)^{2k}}{\Gamma(k+\frac12)\Gamma(k+\frac{p}{2}+1)}\Bigg\}\\
&=&\frac{\lambda e^{-\lambda t}}{2c}(ct)^p\Gamma\left(\frac{p+1}{2}\right)\Bigg\{ct\sum_{k=0}^\infty\left(\frac{\lambda t}{2}\right)^{2k}\frac{1}{\Gamma(k+1)\Gamma(k+\frac{p+1}{2}+1)}\\
&&+\left(\frac{\lambda}{2c}\right)^{-1}\frac{1}{\sqrt{\pi}\Gamma(\frac{p}{2}+1)}+\left(\frac{\lambda}{2c}\right)^{-1}\sum_{k=1}^\infty\left(\frac{\lambda t}{2}\right)^{2k}\frac{1}{\Gamma(k+\frac12)\Gamma(k+\frac{p}{2}+1)}\Bigg\}\\
&=&\frac{\lambda e^{-\lambda t}}{2c}(ct)^p\Gamma\left(\frac{p+1}{2}\right)\Bigg\{ct\sum_{k=0}^\infty\left(\frac{\lambda t}{2}\right)^{2k}\frac{1}{\Gamma(k+1)\Gamma(k+\frac{p+1}{2}+1)}\\
&&+\left(\frac{\lambda}{2c}\right)^{-1}\frac{1}{\sqrt{\pi}\Gamma(\frac{p}{2}+1)}+\left(\frac{\lambda}{2c}\right)^{-1}\sum_{k=0}^\infty\left(\frac{\lambda t}{2}\right)^{2k+2}\frac{1}{\Gamma(k+\frac32)\Gamma(k+\frac{p+1}{2}+\frac32)}\Bigg\}\\
&=&\frac{\lambda e^{-\lambda t}}{2c}(ct)^p\Gamma\left(\frac{p+1}{2}\right)\left\{ct\left(\frac{\lambda t}{2}\right)^{-\frac{p+1}{2}}(I_{\frac{p+1}{2}}(\lambda t)+{\bf L}_{\frac{p+1}{2}}(\lambda t))+\left(\frac{\lambda}{2c}\right)^{-1}\frac{1}{\sqrt{\pi}\Gamma(\frac{p}{2}+1)}\right\}.
\end{eqnarray*}
For $X_1(t)$ one has that
\begin{eqnarray*}
E\{X_1^p(t)\}&=&\frac{2e^{-\lambda t}}{c\sqrt{\lambda \pi t^3}}\sum_{k=0}^{\infty}\left(\frac{\lambda}{c^2 t}\right)^{k+\frac12}\frac{k+1}{\Gamma(k+\frac{3}{2})}\int_0^{ct}x^p(c^2t^2-x^2)^{k+\frac12}dx\\
&=&(x=ct\sqrt{y})\\
&=&\frac{e^{-\lambda t}}{c\sqrt{\lambda \pi t^3}}\sum_{k=0}^{\infty}\left(\frac{\lambda}{c^2 t}\right)^{k+\frac12}\frac{k+1}{\Gamma(k+\frac{3}{2})}(ct)^{p+2k+2}\int_0^1y^{\frac{p-1}{2}}(1-y)^{k+\frac12}dy\\
&=&\frac{e^{-\lambda t}}{c\sqrt{\lambda \pi t^3}}\left(\frac{\lambda}{c^2 t}\right)^{\frac12}\Gamma\left(\frac{p+1}{2}\right)(ct)^{p+2}\sum_{k=0}^{\infty}\left(\lambda t\right)^{k}\frac{k+1}{\Gamma(k+\frac{p}{2}+2)}\\
&=&\frac{e^{-\lambda t}}{\sqrt{\pi}}\Gamma\left(\frac{p+1}{2}\right)(ct)^p\sum_{k=0}^{\infty}\left(\lambda t\right)^{k}\frac{k+\frac p2+1-\frac p2}{\Gamma(k+\frac{p}{2}+2)}\\
&=&\frac{e^{-\lambda t}}{\sqrt{\pi}}\Gamma\left(\frac{p+1}{2}\right)(ct)^p\left\{E_{1,\frac p2+1}(\lambda t)-\frac p2E_{1,\frac p2+2}(\lambda t)\right\}
\end{eqnarray*}
If $p$ is odd, immediately follows that $E\{X_0^p(t)\}=E\{X_1^p(t)\}=0$. 
\end{proof}
\begin{remark}
From \eqref{mom1} we immediately reobtain the expression \eqref{mom2} for $\nu=1$, by setting $p=2$. For $E\{X_0^2(t)\}$, starting from \eqref{intstep} and using the duplication formula, one has that
\begin{eqnarray*}
E\{X_0^2(t)\}&=&e^{-\lambda t}(ct)^2\sum_{k=0}^\infty\frac{(\lambda t)^k}{k!(k+2)}\\
&=&\frac{e^{-\lambda t}}{2c}(ct)^2\Gamma\left(\frac 32\right)\sum_{k=0}^\infty \left(\frac{\lambda}{2c}\right)^{k-1}\frac{(ct)^k}{\Gamma(\frac{k+1}{2})\Gamma(\frac k2+2)}\\
&=&\frac{(ct)^2}{2}\sum_{k=0}^\infty\frac{(\lambda t)^k}{\Gamma(k+1)}\frac{\Gamma(\frac k2+1)}{\Gamma(\frac k2+2)}\\
&=&(ct)^2\sum_{k=0}^\infty\frac{(\lambda t)^k}{\Gamma(k+1)(k+2)}\\
&=&(ct)^2\sum_{k=0}^\infty\frac{(\lambda t)^k(k+2-1)}{\Gamma(k+1)(k+1)(k+2)}\\
&=&e^{-\lambda t}(ct)^2\left\{E_{1,2}(\lambda t)-E_{1,3}(\lambda t)\right\}
\end{eqnarray*}
corresponding to \eqref{mom2} for $\nu=0$.
\end{remark}
\begin{remark}
Since, for $x> 0$ and $\mu\geq 0$, $I_\mu(x)$, ${\bf L}_{\mu}(x)$ and $E_{1,\mu}$ are monotone increasing functions as $x\to\infty$, the following approximations hold as $x\to 0$
\begin{align*}
&I_\mu(x)\sim \frac{x^\mu}{2^\mu\Gamma(\mu+1)},\\
&{\bf L}_{\mu}(x)\sim0,\\
&E_{1,\beta}(x)\sim\frac{1}{\Gamma(\beta)}.
\end{align*}
Therefore, for $\lambda t \to 0$, we have that
\begin{align*}
&E\{X_0^p(t)\}\sim e^{-\lambda t}(ct)^{p}\left\{\frac{\lambda t}{2}+\frac{\Gamma(\frac{p+1}{2})}{\sqrt{\pi}\Gamma(\frac{p}{2}+1)}\right\}\\
&E\{X_1^p(t)\} \sim \frac{e^{-\lambda t}}{\sqrt{\pi}}\Gamma\left(\frac{p+1}{2}\right)(ct)^p\frac{1}{\Gamma(\frac p2+2)}
\end{align*}
\end{remark}
We conclude this Section discussing some connections between the previous results and the random motions moving on a Non-Euclidean plane. The upper half-plane $H_2^+=\{(x,y):y>0,x\in \mathbb{R}\}$ endowed with the metric
$$\frac{\sqrt{dx^2+dy^2}}{y}$$
is a model of Non-Euclidean (hyperbolic) space. The geodesic curves in this space are eiher the vertical half lines or half-circles whose centers lie on the $x$-axis. Similarly to Orsingher and De Gregorio (2007a), we consider a motion 
\begin{equation}
Y_\nu(t)=e^{X_\nu(t)}
\end{equation}
developing on the $y$-axis of the space $H_2^+$, starting from the origin $O=(0,1)$ at time $t=0$. The probability distribution of $Y_m(t)$ with $m=0,1,$ is equal to
\begin{equation*}
p^m(\log y,t)\frac 1y,\quad y>0
\end{equation*}
and by means of Corollary \ref{corcf}, the conditional mean values of $Y_0(t)$ and $Y_1(t)$ become
$$E\left\{Y_0(t)| N(t)=n\right\}=E\left\{e^{\beta X_0(t)}|N(t)=n\right\}\Big|_{\beta=1}=\frac{\Gamma\left(\frac n2+1\right)2^{\frac n2}}{(ct)^{\frac n2}}I_{\frac n2}( ct)
$$
$$E\left\{Y_1(t)| N(t)=n\right\}=E\left\{e^{\beta X_1(t)}|N(t)=n\right\}\Big|_{\beta=1}=\frac{2^{n+1}\Gamma(n+2)}{(ct)^{n+1}}I_{n+1}(  ct)$$
Furthermore, by considering at time $t$ the hyperbolic distance $\eta_\nu(t)$ from the origin $O$ of $Y_\nu(t)$, we have that (see Orsingher and De Gregorio, 2007a)
\begin{equation}
\eta_\nu(t)=\int_{\min(1,Y_\nu(t))}^{\max(1,Y_\nu(t))}\frac{dy}{y}=|X_\nu(t)|
\end{equation}
We are able to obtain a lower bound for the distribution function of $\eta_\nu(t)$. Indeed, we get that
\begin{eqnarray}
P(\eta_\nu(t)<\eta)&=&P(|X_\nu(t)|<\eta)\notag\\
&\geqslant&1-\frac{1}{\eta^2}E\{X^2_\nu(t)\}\notag
\end{eqnarray}
where $E\{X^2_\nu(t)\}$ is given by the formula \eqref{mom2}. Clearly, for $\nu=0,1,$ we obtain the exact expression of the distribution function of the hyperbolic distance, namely 
$$P(\eta_m(t)<\eta)=2\int_0^{\eta}p^m(x,t)dx$$
with $0<\eta\leq ct$ and $m=0,1$.

\section{Randomly varying time stochastic motions}\label{sec:comp}
So far, we have analyzed a random motion $X_\nu(t)$ evolving up to no-random time $t>0$, deriving its exact probability distribution in two particular cases. In this Section, we focus our attention on the random motion $X_m(t),\,m=0,1$, defined as in \eqref{gentel}, with randomly varying time. 

In order to develop our analysis , we take into account families of random times which include some well-known random variables. In particular, we consider Bessel and Gamma processes as random times. Our choice falls on these two processes because at this way, we are able to include a wide range of probability distributions often used to model the time in many theoretical and real situations.   Clearly, every random variable successively used as random clock will be supposed independent from $X_m(t)$ and $N(t)$. Therefore, we analyze the effect due to the composition of these random times with the random motion $X_m(t)$ on the related density laws. 

\subsection{Random times involving Brownian motions}
 Let us consider a Bessel process starting from zero
$$R^d(t)=\sqrt{\sum_{i=1}^d B_i^2(t)},\quad t>0,\, d\geq 1,$$
where $B_i(t)$s are independent standard Brownian motions. It is well-known that the probability density law of $R^d(t)$ is equal to 
\begin{equation}\label{densBessst}
f^d(r)=\frac{1}{\Gamma(\frac d2)}\frac{r^{d-1}}{2^{\frac d2-1} t^{\frac d2}}e^{-\frac{r^2}{2t}},\quad r > 0.
\end{equation}
  At time $t$, we deal with a random motion $X_m(t),\,m=0,1$, with a Bessel random time $R^d(t)$. Since $X_m(R^d(t))$ is located inside $(-cR^d(t),cR^d(t))$, its support is the whole real line. Recalling that $v_0=\frac n2$, $v_1=n+1$ and by indicating with $B(a,b)=\frac{\Gamma(a)\Gamma(b)}{\Gamma(a+b)},\,a>0,\, b>0$ a Gamma function and with $B(t)$ a standard Brownian motion at time $t$, we are able to provide the following theorem.
\begin{theorem}\label{besselteo}
Given $N(t)=n,$ with $n\geq 1$, such that $v_m>\frac d2-1$, we have the following conditional distribution
\begin{eqnarray}\label{stopdist}
P\left\{X_m(R^d(t))\in dx|N(t)=n\right\}=\frac{dx}{B\left(\frac d2,v_m-\frac d2+1\right)}\int_0^1w^{\frac d2-1}(1-w)^{v_m-\frac d2}\frac{e^{-\frac{x^2}{2c^2t w}}}{\sqrt{2\pi tw}c}dw
\end{eqnarray}
with $x\in \mathbb{R}$.
\end{theorem}
\begin{proof} By using a similar approach to that adopted by Beghin and Orsingher (2009), we can write that
\begin{eqnarray*}
&&P\left\{X_m(R^d(t))\in dx|N(t)=n\right\}\\
&&=\int_0^\infty P\left\{X_m(R^d(t))\in dx|N(t)=n,R^d(t)=s\right\}P\{R^d(t)\in  ds\}\\
&&=\frac{dx}{2\pi}\frac{\Gamma(v_m+1)\Gamma(v_m)}{\Gamma(2v_m)}\frac{1}{\Gamma(\frac d2)2^{\frac d2-1}t^{\frac d2}}\int_0^\infty\left(\frac{2}{c s}\right)^{v_m}(c^2s^2-x^2)^{v_m-\frac12}{\bf 1}_{\{|x|<cs\}}s^{d-1}e^{-\frac{s^2}{2t}}ds
\end{eqnarray*}
Therefore, by means of Corollary \ref{corcf} and for any $n\geq 1$ such that $v_m>\frac d2-1$, we are able to explicit the Fourier transform of $X_m(R^d(t))$ as follows 
\begin{eqnarray*}
E\left\{e^{i\alpha X_m(R^d(t))}|N(t)=n\right\}&=&\int_{-\infty}^{+\infty}e^{i\alpha x}P\left\{X_m(R(t))\in dx|N(t)=n\right\}\\
&=&\Gamma(v_m+1)\left(\frac{2}{\alpha c}\right)^{v_m}\frac{1}{\Gamma(\frac d2)2^{\frac d2-1}t^{\frac d2}}\int_0^\infty J_{v_m}(\alpha cs)s^{d-v_m-1}e^{-\frac{s^2}{2t}}ds\\
&=&\frac{\Gamma(v_m+1)}{\Gamma(\frac d2)2^{\frac d2-1}t^{\frac d2}}\sum_{k=0}^\infty\frac{(-1)^k}{k!\Gamma(k+v_m+1)}\left(\frac{\alpha c}{2}\right)^{2k}\int_0^\infty s^{2k+d-1} e^{-\frac{s^2}{2t}}ds\\
&=&\left(y=\frac{s^2}{2t}\right)\\
&=&\frac{\Gamma(v_m+1)}{\Gamma(\frac d2)}\sum_{k=0}^\infty\frac{(-1)^k}{k!\Gamma(k+v_m+1)}\left(\frac{\alpha^2 c^2 t}{2}\right)^{k}\int_0^\infty y^{k+\frac d2-1} e^{- y}dy\\
&=&\frac{\Gamma(v_m+1)}{\Gamma(\frac d2)}\sum_{k=0}^\infty\frac{(-1)^k\Gamma(k+\frac d2)}{k!\Gamma(k+v_m+1)}\left(\frac{\alpha^2 c^2 t}{2}\right)^{k}\\
&=&\frac{\Gamma(v_m+1)}{\Gamma(\frac d2)\Gamma\left(v_m-\frac d2+1\right)}\sum_{k=0}^\infty\frac{(-1)^k}{k!}B\left(k+\frac d2,v_m-\frac d2+1\right)\left(\frac{\alpha^2 c^2 t}{2}\right)^{k}\\
&=&\frac{1}{B\left(\frac d2,,v_m-\frac d2+1\right)}\sum_{k=0}^\infty\frac{(-1)^k}{k!}\left(\frac{\alpha^2 c^2 t}{2}\right)^{k}\int_0^1w^{k+\frac d2-1}(1-w)^{v_m-\frac d2}dw\\
&=&\frac{1}{B\left(\frac d2,v_m-\frac d2+1\right)}\int_0^1w^{\frac d2-1}(1-w)^{v_m-\frac d2}e^{-\frac{\alpha^2 c^2tw}{2}}dw
\end{eqnarray*}
Finally, by inverting $E\left\{e^{i\alpha X_m(R^d(t))}|N(t)=n\right\}$ the result \eqref{stopdist} emerges.  \end{proof}

The probability \eqref{stopdist} claims that, conditionally on the number of Poisson events such that $v_m>\frac d2-1$, the random process $X_m(R^d(t))$, is distributed as a centered Gaussian with variance $c^2tW$, where $W\sim B(\frac d2,v_m-\frac d2+1)$.
\begin{remark}\label{rem}
By using the same approach of the previous proof and bearing in mind the Theorem \ref{chf}, under the condition $\nu>\frac{d}{2}-1$, we get that 
\begin{align}\label{stopdist0}
&P\left\{X_\nu(R^d(t))\in dx|N(t)=0\right\}=\frac{dx}{B\left(\frac d2,\nu-\frac d2+1\right)}\int_0^1w^{\frac d2-1}(1-w)^{\nu-\frac d2}\frac{e^{-\frac{x^2}{2c^2t w}}}{\sqrt{2\pi tw}c}dw
\end{align} 
In particular, we are interested to the random motions obtained by setting $\nu=0$ and $\nu=1$.
It is clear that for $\nu=0$, the condition $\nu>\frac d2-1$ is satisfied only for $d=1$, whilst for $d=2$ one has
\begin{align}
&P\left\{X_0(R^2(t))\in dx|N(t)=0\right\}=dx\frac{e^{-\frac{x^2}{2c^2t }}}{\sqrt{2\pi t}c}.
\end{align} 
If $\nu=1$, the above condition and the representation \eqref{stopdist0} hold for both $d=1$ and $d=2$.
\end{remark}
From \eqref{densBessst}, we can derive some well-known probability distributions. Indeed, for $d=1$, we get $R^1(t)=|B(t)|$, that is a reflected Brownian motion around the $x$-axis and the its density law becomes $f^1(r)=\frac{\sqrt{2}}{\sqrt{\pi t}}e^{-\frac{r^2}{2t}}$. Furthermore, $|B(t)|$ represents the Brownian time used in the definition of the iterated Brownian motion (see, Allouba, 2002). For $d=2$, we obtain a Rayleigh random variable with $f^2(r)=\frac{r}{t}e^{-\frac{r^2}{2t}}$, which also emerges analyzing the distribution of the maximum of a Brownian bridge. Moreover, the probability \eqref{stopdist} holds for each $n\geq 1$, being the conditions $v_m>-\frac12\, (d=1)$ and $v_m>0\, (d=2)$ always satisfied. Actually, for the process $X_1$, the representation \eqref{stopdist} also yields when $d=3,4,5$. Nevertheless, we restrict us to the cases $d=1,2$.  Therefore, we are able to explicit the unconditional probability distributions for $X_m(|B(t)|)$ and $X_m(R^2(t))$.

\begin{theorem}\label{cor} 
For $d=1$, the following probability yields
\begin{align}\label{uncstopdist0}
P\left\{X_m(|B(t)|)\in dx\right\}=dx\int_{-ct}^{ct}\frac{\sqrt{t}e^{-\frac{x^2 t}{2y^2}}}{\sqrt{2\pi y^2}}p^m(y,t)dy,\quad x\in \mathbb{R},
\end{align}
where $p^0(y,t)$  and $p^1(y,t)$ are defined respectively by \eqref{law} and \eqref{law2}. Furthermore, for $d=2$, one has
\begin{align}\label{uncstopdist}
&P\left\{X_0(R^2(t))\in dx\right\}=dx e^{-\lambda t}\left\{\frac{\lambda t}{2}\int_0^1\frac{e^{\lambda t\sqrt{1-w}}}{\sqrt{1-w}}\frac{e^{-\frac{x^2}{2c^2t w}}}{\sqrt{2\pi tw}c}dw+\frac{e^{-\frac{x^2}{2c^2t }}}{\sqrt{2\pi t}c}\right\}\\
&P\left\{X_1(R^2(t))\in dx\right\}=dx\int_0^1e^{-\lambda tw}[1+\lambda t(1-w)]\frac{e^{-\frac{x^2}{2c^2t w}}}{\sqrt{2\pi tw}c}dw\label{uncstopdist2}
\end{align}
\end{theorem}
\begin{proof} By taking into account \eqref{stopdist} and \eqref{stopdist0},
if $m=0\, (v_0=\frac n2)$ and $d=1$, we get that
\begin{eqnarray*}
\frac{1}{dx}P\left\{X_0(|B(t)|)\in dx\right\}&=&e^{-\lambda t}\sum_{n=0}^\infty\frac{(\lambda t)^n}{n!}\frac{1}{B(\frac12,\frac{n+1}{2})}\int_0^1w^{-\frac12}(1-w)^{\frac{n-1}{2}}\frac{e^{-\frac{x^2}{2c^2t w}}}{\sqrt{2\pi t w}c}dw\\
&=&\frac{e^{-\lambda t}}{\sqrt{\pi}}\int_0^1w^{-\frac12}(1-w)^{-\frac{1}{2}}\frac{e^{-\frac{x^2}{2c^2t w}}}{\sqrt{2\pi tw}c}\sum_{n=0}^\infty\frac{(\lambda t)^n}{n!}(1-w)^{\frac n2}\frac{\Gamma(\frac n2+1)}{\Gamma(\frac{n+1}{2})}dw\\
&=&e^{-\lambda t}\int_0^1w^{-\frac12}(1-w)^{-\frac{1}{2}}\frac{e^{-\frac{x^2}{2c^2t w}}}{\sqrt{2\pi tw}c}\sum_{n=0}^\infty\frac{1}{\Gamma^2(\frac{n+1}{2})}\left(\frac{\lambda t}{2}\sqrt{1-w}\right)^ndw\\
&=&\frac{\lambda te^{-\lambda t}}{2}\int_0^1w^{-\frac12}\frac{e^{-\frac{x^2}{2c^2t w}}}{\sqrt{2\pi tw}c}\sum_{n=0}^\infty\frac{1}{\Gamma^2(\frac{n+1}{2})}\left(\frac{\lambda t}{2} \sqrt{1-w}\right)^{n-1}dw\\
&=&(y=ct\sqrt{w})\\
&=&\frac{\lambda e^{-\lambda t}}{c}\int_0^{ct}\frac{\sqrt{t}e^{-\frac{tx^2}{2y^2 }}}{\sqrt{2\pi y^2}}\sum_{n=0}^\infty\frac{1}{\Gamma^2(\frac{n+1}{2})}\left(\frac{\lambda}{2c}  \sqrt{c^2t^2-y^2}\right)^{n-1}dy\\
&=&\int_{-ct}^{ct}\frac{\sqrt{t}e^{-\frac{tx^2}{2y^2 }}}{\sqrt{2\pi y^2}}\frac{\lambda e^{-\lambda t}}{2c}\sum_{n=0}^\infty\frac{1}{\Gamma^2(\frac{n+1}{2})}\left(\frac{\lambda}{2c}  \sqrt{c^2t^2-y^2}\right)^{n-1}dy
\end{eqnarray*}
For $m=1\, (v_1=n+1)$ and $d=1$, we have that
\begin{eqnarray*}
\frac{1}{dx}P\left\{X_1(|B(t)|)\in dx\right\}&=&e^{-\lambda t}\sum_{n=0}^\infty\frac{(\lambda t)^n}{n!}\frac{1}{B(\frac12,n+\frac{3}{2})}\int_0^1w^{-\frac12}(1-w)^{n+\frac{1}{2}}\frac{e^{-\frac{x^2}{2c^2t w}}}{\sqrt{2\pi tw}c}dw\\
&=&\frac{e^{-\lambda t}}{\sqrt{\pi}}\int_0^1w^{-\frac12}\frac{e^{-\frac{x^2}{2c^2t w}}}{\sqrt{2\pi tw}c}\sum_{n=0}^\infty\frac{(\lambda t)^n}{n!}(1-w)^{n+\frac 12}\frac{\Gamma(n+2)}{\Gamma(n+\frac{3}{2})}dw\\
&=&\frac{ e^{-\lambda t}}{\sqrt{\pi\lambda t}}\int_0^1w^{-\frac12}\frac{e^{-\frac{x^2}{2c^2t w}}}{\sqrt{2\pi tw}c}\sum_{n=0}^\infty\frac{n+1}{\Gamma(n+\frac{3}{2})}\left(\lambda t(1-w)\right)^{n+\frac12}dw\\
&=&(y=ct\sqrt{w})\\
&=&\frac{2}{c}\frac{e^{-\lambda t}}{\sqrt{\lambda \pi t^3}}\int_0^{ct}\frac{\sqrt{t}e^{-\frac{x^2 t}{2y^2}}}{\sqrt{2\pi y^2}}\sum_{n=0}^\infty\frac{n+1}{\Gamma(n+\frac{3}{2})}\left(\frac{\lambda}{c^2t}(c^2t^2-y^2)\right)^{n+\frac12}dy\\
&=&\int_{-ct}^{ct}\frac{\sqrt{t}e^{-\frac{x^2 t}{2y^2}}}{\sqrt{2\pi y^2}}\frac{e^{-\lambda t}}{c\sqrt{\lambda \pi t^3}}\sum_{n=0}^\infty\frac{n+1}{\Gamma(n+\frac{3}{2})}\left(\frac{\lambda}{c^2t}(c^2t^2-y^2)\right)^{n+\frac12}dy
\end{eqnarray*}
Therefore, the result \eqref{uncstopdist0} is proved. Developing the quantity
$$e^{-\lambda t}\sum_{n=1}^\infty\frac{(\lambda t)^n}{n!}P\left\{X_m(R^2(t))\in dx|N(t)=n\right\}$$
for $m=0,1$, and by taking into account the Remark \ref{rem}, it is not hard to prove the results \eqref{uncstopdist}, \eqref{uncstopdist2}. 
\end{proof}
We point out that the probability distribution \eqref{uncstopdist0} says us that the process $X_m(t)$ stopped at reflected Brownian time is equivalent in distribution to a Brownian motion with variance $\frac1tX_m^2(t)$, i.e.
\begin{equation}\label{equal}
X_m(|B(t)|)\stackrel{d}{=}B\left(\frac1tX_m^2(t)\right)
\end{equation} 
\begin{remark}
Recalling that for a centered Gaussian with variance $\sigma^2$ the moments are given by
$$\sigma^p\frac{\Gamma(p+1)}{2^{\frac p2}\Gamma(\frac p2+1)}$$
if $p$ is even, while are 0 if $p$ is odd, we obtain that
$$E\{X_m^p(|B(t)|)\}=\frac{\Gamma(p+1)}{(2t)^{\frac p2}\Gamma(\frac p2+1)}E\{X_m^p(t)\}=E\left\{B^p\left(\frac{1}{t}\right)\right\}E\{X_m^p(t)\}$$
where $E\{X_m^p(t)\}$ is defined as in Theorem \ref{momp}. 
\end{remark}

The result \eqref{equal} is more general. Indeed,
let $\mathcal{N}(t)$ be a Gaussian process with mean 0 and variance $\sigma^2(t)$, it is not hard to prove by using the same argument adopted in the proof of Theorem \ref{besselteo} and Corollary \ref{cor}, that $X_m(|\mathcal{N}(t)|)$ is distributed as a Gaussian random variable with variance $\frac{\sigma^2(t)}{t^2}X_m^2(t)$. In other words, the following distributional relationship holds 
\begin{equation*}
X_m(|\mathcal{N}(t)|)\stackrel{d}{=}B\left(\frac{\sigma^2(t)}{t^2}X_m^2(t)\right).
\end{equation*} 
For example, if:
\begin{itemize}
\item $\mathcal{N}(t)=B^H(t)$, that is a fractional Brownian motion with Hurst index $H\in(0,1)$, we have that $\sigma^2(t)=t^{2H}$ and then
$X_m(|B^H(t)|)\stackrel{d}{=}B\left(t^{2H-2}X_m^2(t)\right)
$, which contains as particular case the result \eqref{equal} for $H=\frac 12$;
\item $\mathcal{N}(t)=\int_0^th(s)dB(s)$, where $h(s)$ is a well-defined deterministic function,  the variance is $\sigma^2(t)=\int_0^th^2(s)ds$ and then $X_m(|\int_0^th(s)dB(s)|)\stackrel{d}{=}B\left(\frac{\int_0^th^2(s)ds}{t^2}X_m^2(t)\right)
$. Clearly, for $h(s)=1$, we reobtain the equality \eqref{equal};
\item $\mathcal{N}(t)=\int_0^tB(s)ds$, the variance is given by $\sigma^2(t)=\frac{t^3}{3}$, therefore we obtain that $X_m(\int_0^tB(s)ds)\stackrel{d}{=}B(\frac t3X_m^2(t))$.
\end{itemize}

 Theorem \ref{besselteo} can also be generalized by dealing with an $l$-times interated Bessel process, namely
  $$\mathcal{R}_l^d(t)=R_1^d(R_2^d(...(R_{l+1}^d(t)...))),\quad t>0,$$
 where $R_j^d$s, $j=1,2,...,l$, are independent Bessel processes. The random process $\mathcal{R}_l^d(t)$ has density law given by
 $$f_l^d(r)=\frac{1}{(\Gamma(\frac d2)2^{\frac d2-1})^{l+1}}\int_0^\infty\cdot\cdot\cdot\int_0^\infty\frac{r^{d-1}e^{-\frac{r^2}{2t_1}}}{ t_1^{\frac d2}}\frac{t_1^{d-1}e^{-\frac{t_1^2}{2t_2}}}{t_2^{\frac d2}}\cdot\cdot\cdot \frac{t_l^{d-1}e^{-\frac{t_l^2}{2t}}}{t^{\frac d2}}dt_1dt_2\cdots dt_l$$ 
and leads to the next result.
 
 \begin{theorem}\label{genstopdist}
 Given $N(t)=n$, with $n\geq 1$, such that $v_m>\frac d2-1$, we have that
 \begin{eqnarray}
&& P\left\{X_m(\mathcal{R}_l^d(t))\in dx|N(t)=n\right\}\\
 &&=\frac{dx}{B(\frac d2,v_m-\frac d2+1)}\frac{1}{(\Gamma(\frac d2)2^{(\frac d2-1)})^l}
 \int_0^\infty\frac{t_1^{d-1}e^{-\frac{t_1^2}{2t_2}}}{t_2^{\frac d2}}dt_1\int_0^\infty\frac{t_2^{d-1}e^{-\frac{t_2^2}{2t_3}}}{t_3^{\frac d2}}dt_2\cdot\cdot\cdot\int_0^\infty\frac{t_l^{d-1}e^{-\frac{t_l^2}{2t}}}{t^{\frac d2}}dt_l\notag\\
 &&\quad \times\int_0^1w^{\frac d2-1}(1-w)^{v_m-\frac d2}\frac{e^{-\frac{x^2}{2c^2t_1 w}}}{\sqrt{2\pi t_1w}c}dw\notag
 \end{eqnarray}
 with $x\in \mathbb{R}$.
 \end{theorem}
 \begin{proof} Since $X_m(\mathcal{R}_l^d(t))$ has support on the interval $(-c\mathcal{R}_l^d(t),c\mathcal{R}_l^d(t))$ and then on $\mathbb{R}$, we can write that
\begin{eqnarray*}
&&P\left\{X_m(\mathcal{R}_l^d(t))\in dx|N(t)=n\right\}\\
&&=\int_0^\infty P\left\{X_m(\mathcal{R}_l^d(t))\in dx|N(t)=n,\mathcal{R}_l^d(t)=s\right\}P\{\mathcal{R}_l^d(t)\in  ds\}\\
&&=\frac{dx}{2\pi}\frac{\Gamma(v_m+1)\Gamma(v_m)}{\Gamma(2v_m)}\int_0^\infty\left(\frac{2}{c s}\right)^{v_m}(c^2s^2-x^2)^{v_m-\frac12}{\bf 1}_{\{|x|<cs\}}s^{d-1}f_l^d(s)ds.
\end{eqnarray*}
The conditional characteristic function becomes
\begin{eqnarray*}
&&E\left\{e^{i\alpha X_m(\mathcal{R}_l^d(t))}|N(t)=n\right\}\\
&&=\Gamma(v_m+1)\int_0^\infty\left(\frac{2}{\alpha cs}\right)^{v_m}J_{v_m}(\alpha cs)f_l^d(s)ds\\
&&=\frac{\Gamma(v_m+1)}{(\Gamma(\frac d2)2^{\frac d2-1})^{l+1}}\int_0^\infty dt_1 \cdot\cdot\cdot\int_0^\infty dt_l\left(\prod_{j=2}^{l+1}\frac{t_{j-1}^{d-1}}{t_j^{\frac d2}}e^{-\frac{t_{j-1}^2}{2t_j}} \right)\\
&&\quad \times\frac{1}{t_1^{\frac d2}}\sum_{k=0}^\infty\frac{(-1)^k}{k!\Gamma(k+v_m+1)}\left(\frac{\alpha c}{2}\right)^{2k}\int_0^\infty s^{2k+d-1} e^{-\frac{ s^2}{2t_1}}ds\\
&&=\left(y=\frac{s^2}{2t_1}\right)\\
&&=\frac{\Gamma(v_m+1)}{(\Gamma(\frac d2))^{l+1}2^{(\frac d2-1)l}}\int_0^\infty dt_1 \cdot\cdot\cdot\int_0^\infty dt_l\left(\prod_{j=2}^{l+1}\frac{t_{j-1}^{d-1}}{t_j^{\frac d2}}e^{-\frac{t_{j-1}^2}{2t_j}} \right)\sum_{k=0}^\infty\frac{(-1)^k\Gamma(k+\frac d2)}{k!\Gamma(k+v_m+1)}\left(\frac{\alpha^2 c^2t_1}{2}\right)^{k}\\
&&=\frac{\Gamma(v_m+1)}{(\Gamma(\frac d2))^{l+1}2^{(\frac d2-1)l}}\int_0^\infty dt_1 \cdot\cdot\cdot\int_0^\infty dt_l\left(\prod_{j=2}^{l+1}\frac{t_{j-1}^{d-1}}{t_j^{\frac d2}}e^{-\frac{t_{j-1}^2}{2t_j}} \right)\sum_{k=0}^\infty\frac{(-1)^kB(k+\frac d2,v_m-\frac d2+1)}{k!\Gamma(v_m-\frac d2+1)}\left(\frac{\alpha^2 c^2t_1}{2}\right)^{k}\\
&&=\frac{\Gamma(v_m+1)}{\Gamma(v_m-\frac d2+1)(\Gamma(\frac d2))^{l+1}2^{(\frac d2-1)l}}\int_0^\infty dt_1 \cdot\cdot\cdot\int_0^\infty dt_l\left(\prod_{j=2}^{l+1}\frac{t_{j-1}^{d-1}}{t_j^{\frac d2}}e^{-\frac{t_{j-1}^2}{2t_j}} \right)\\
&&\quad \times\int_0^1w^{\frac d2-1}(1-w)^{v_m-\frac d2}e^{-\frac{\alpha^2 c^2t_1w}{2}}dw
\end{eqnarray*}
with $t_{l+1}=t$. By inverting the so-obtained  characteristic  function the proof is completed.
 \end{proof}
Similarly to the simple Bessel time, conditionally on $N(t)=n$, such that the constraint $v_m>\frac d2-1$ is satisfied, $X_m(\mathcal{R}_l^d(t)$  is distributed as a Gaussian random variable with variance given by $c^2R_1^d(R_2^d(\cdots (R_l^d(t))\cdots))W$, $W\sim B(\frac d2,v_m-\frac d2+1)$.
 
 Obviously, also the result contained in the Remark \ref{rem} can be generalized as well. Indeed, by means of the same approach used in the proof of the Theorem \ref{genstopdist}, if $\nu>\frac d2-1$, we obtain that
 \begin{eqnarray}\label{genstopdist02}
&& P\left\{X_\nu(\mathcal{R}_l^d(t))\in dx|N(t)=0\right\}\\
 &&=\frac{dx}{B(\frac d2,\nu-\frac d2+1)}\frac{1}{(\Gamma(\frac d2)2^{(\frac d2-1)})^l}
 \int_0^\infty\frac{t_1^{d-1}e^{-\frac{t_1^2}{2t_2}}}{t_2^{\frac d2}}dt_1\int_0^\infty\frac{t_2^{d-1}e^{-\frac{t_2^2}{2t_3}}}{t_3^{\frac d2}}dt_2\cdot\cdot\cdot\int_0^\infty\frac{t_l^{d-1}e^{-\frac{t_l^2}{2t}}}{t^{\frac d2}}dt_l\notag\\
 &&\quad \times\int_0^1w^{\frac d2-1}(1-w)^{\nu-\frac d2}\frac{e^{-\frac{x^2}{2c^2t_1 w}}}{\sqrt{2\pi t_1w}c}dw.\notag
 \end{eqnarray}
For $\nu=0$ and $d=2$, it is easy to show that
  \begin{eqnarray}\label{genstopdist0}
&&P\left\{X_0(\mathcal{R}_l^2(t))\in dx|N(t)=0\right\}\\
&&=dx\int_0^\infty\frac{t_1e^{-\frac{t_1^2}{2t_2}}}{t_2}dt_1\int_0^\infty\frac{t_2e^{-\frac{t_2^2}{2t_3}}}{t_3}dt_2\cdot\cdot\cdot\int_0^\infty\frac{t_le^{-\frac{t_l^2}{2t}}}{t}dt_l\frac{e^{-\frac{x^2}{2c^2t_1}}}{\sqrt{2\pi t_1}}.\notag
\notag
\end{eqnarray}
By setting $d=1$, $\mathcal{R}_l^d(t)$  becomes an $l$-iterated Brownian motion , namely $\mathcal{R}_l^1(t)=|B_1(|B_2(|\cdots(|B_{l+1}(t)|)\cdots|)|)|$, where $B_j$s are independent Brownian motions. Then, for each $n\geq 1$, we get that
 \begin{eqnarray}
 &&P\left\{X_m(\mathcal{R}_l^1(t))\in dx|N(t)=n\right\}\\
 &&=\frac{dx2^{\frac l2}}{B(\frac 12,v_m+\frac 12)}
 \int_0^\infty\frac{e^{-\frac{t_1^2}{2t_2}}}{\sqrt{\pi t_2}}dt_1\int_0^\infty\frac{e^{-\frac{t_2^2}{2t_3}}}{\sqrt{\pi t_3}}dt_2\cdot\cdot\cdot\int_0^\infty\frac{e^{-\frac{t_l^2}{2t}}}{\sqrt{\pi t}}dt_l\notag\\
 &&\quad \times\int_0^1w^{-\frac 12}(1-w)^{v_m-\frac 12}\frac{e^{-\frac{x^2}{2c^2t_1 w}}}{\sqrt{2\pi t_1w}c}dw.\notag
 \end{eqnarray}
 Furthermore, after some calculations similar to those of the proof of Theorem \ref{cor}, we are able to explicit the following unconditonal distribution 
   \begin{eqnarray}
 &&P\left\{X_m(|B_1(|B_2(|\cdots(|B_{l+1}(t)|)\cdots)|)|)|)\in dx\right\}\\
 &&=dx
 2^{\frac l2}\int_0^\infty\frac{e^{-\frac{t_1^2}{2t_2}}}{\sqrt{\pi t_2}}dt_1\int_0^\infty\frac{e^{-\frac{t_2^2}{2t_3}}}{\sqrt{\pi t_3}}dt_2\cdot\cdot\cdot\int_0^\infty\frac{e^{-\frac{t_l^2}{2t}}}{\sqrt{\pi t}}dt_l\int_{-ct}^{ct}\frac{te^{-\frac{x^2t^2}{2t_1y^2}}}{\sqrt{2\pi t_1y^2}}p^m(y,t)dy\notag
 \end{eqnarray}
which allows us to claim that 
$$X_m(\mathcal{R}_l^1(t))\stackrel{d}{=}B\left(|B_1(|\cdots(|B_{l}(t)|)\cdots)|)|\frac{X_m^2(t)}{t^2}\right).$$

Now, we analyze the effect due to the random clock defined as the time spent on the positive axis (sojourn time) by a standard Brownian motion $B(t),\,t>0$, namely
$$\Gamma(t)=\int_0^t{\bf 1}_{\{B(s)>0\}}(s)ds.$$ 
The density function of $\Gamma(t)$ is $\gamma(s)=\frac{1}{\pi\sqrt{s(t-s)}},\,0<s<t$, that it is also known as arcsin law. We have the following result concerning $X_m(\Gamma(t))$.
\begin{theorem} For $n\geq 1$, the following probability distribution holds
\begin{eqnarray}\label{teosouj}
P\{X_m(\Gamma(t))\in dx|N(t)=n\}=\frac{dx}{B(1,v_m)\pi} \int_0^t\frac{ds}{\sqrt{s(t-s)}}\int_0^1(1-w)^{v_m-1}dw\frac{1}{2\pi}\frac{1}{\sqrt{c^2s^2w-x^2}}.
\notag
\end{eqnarray}
\end{theorem}
\begin{proof} In this case
$$P\{X_m(\Gamma(t))\in dx|N(t)=n\}=\frac{dx}{2\pi}\frac{\Gamma(v_m+1)\Gamma(v_m)}{\Gamma(2v_m)}\int_0^t\left(\frac{2}{ cs}\right)^{v_m}(c^2s^2-x^2)^{v_m-\frac12}\gamma(s)ds$$
As done so far, we consider the Fourier transform for $X_m(\Gamma(t))$. Hence, we get
\begin{eqnarray*}
E\{e^{i\alpha X_m(\Gamma(t))}|N(t)=n\}&=&\frac{\Gamma(v_m+1)}{\pi}\sum_{k=0}^\infty\frac{(-1)^k}{k!\Gamma(k+v_m+1)}\left(\frac{\alpha c}{2}\right)^{2k}\int_0^ts^{2k-\frac12}(t-s)^{-\frac12}ds\\
&=&(s=ty)\\
&=&\frac{\Gamma(v_m+1)}{\pi}\sum_{k=0}^\infty\frac{(-1)^k}{k!\Gamma(k+v_m+1)}\left(\frac{\alpha ct}{2}\right)^{2k}\int_0^1y^{2k-\frac12}(1-y)^{-\frac12}dy\\
&=&\frac{\Gamma(v_m+1)}{\sqrt{\pi}}\sum_{k=0}^\infty\frac{(-1)^k\Gamma(2k+\frac12)}{k!\Gamma(k+v_m+1)\Gamma(2k+1)}\left(\frac{\alpha ct}{2}\right)^{2k}\\
&=&\frac{v_m}{\pi}\sum_{k=0}^\infty\frac{(-1)^k\Gamma(2k+\frac12)\Gamma(k+1)\Gamma(v_m)\Gamma(\frac12)}{k!\Gamma(k+1)\Gamma(k+v_m+1)\Gamma(2k+1)}\left(\frac{\alpha ct}{2}\right)^{2k}\\
&=&\frac{v_m}{\pi}\sum_{k=0}^\infty\frac{(-1)^k}{(k!)^2}B\left(2k+\frac12,\frac12\right)B\left(k+1,v_m\right)\left(\frac{\alpha ct}{2}\right)^{2k}\\
&=&\frac{v_m}{\pi} \int_0^1z^{-\frac12}(1-z)^{-\frac12}dz\int_0^1(1-w)^{v_m-1}dwJ_0(\alpha ct z\sqrt{w})
\end{eqnarray*}
Therefore, by taking into account the Theorem \ref{chf}-\ref{condlaw}, we obtain that
\begin{eqnarray*}
\frac{1}{dx}P\{X_m(\Gamma(t))\in dx|N(t)=n\}&=&\frac{v_m}{\pi} \int_0^1z^{-\frac12}(1-z)^{-\frac12}dz\int_0^1(1-w)^{v_m-1}dw\frac{1}{2\pi}\frac{1}{\sqrt{c^2t^2z^2w-x^2}}\\
&=&(s=tz)\\
&=&\frac{v_m}{\pi} \int_0^t\frac{ds}{\sqrt{s(t-s)}}\int_0^1(1-w)^{v_m-1}dw\frac{1}{2\pi}\frac{1}{\sqrt{c^2s^2w-x^2}}\end{eqnarray*}
\end{proof}
From Theorem \eqref{teosouj} we conclude that the random motions $X_m(t)$ with random time $\Gamma(t)$ is equivalent in distribution, to the random motion $X_0(t)$ with a random time given by $\Gamma(t)\sqrt{W}$, where $W\sim B(1,v_m)$, which maintains the velocity initially chosen until $t$. Then, one has that

$$P\{X_m(\Gamma(t))\in dx|N(t)=n\}=P\{X_0(\Gamma(t)\sqrt{W})\in dx|N(t)=0\}.$$

\begin{remark}
The random process $\Gamma(t)$ is also connected with $X_m(R^d(t))$. Indeed, recalling that $$P\{\Gamma(t)\in ds|B(t)>0\}=\frac{2}{\pi t}\frac{\sqrt{s}}{\sqrt{(t-s)}}ds, \,P\{\Gamma(t)\in ds|B(t)<0\}=\frac{2}{\pi t}\frac{\sqrt{t-s}}{\sqrt{s}}ds,$$ $0<s<t$, from \eqref{stopdist}  and \eqref{stopdist0} the following equalities hold:
\begin{eqnarray*}
P\{X_0(|B_1(t)|)\in dx|N(t)=0\}&=&P\{B_1(c^2\Gamma(t))\in dx\}\\
&=&\frac{dx}{\pi}\int_0^t\frac{1}{\sqrt{s(t-s)}}\frac{e^{-\frac{x^2}{2c^2 s}}}{\sqrt{2\pi s}c}ds,
\end{eqnarray*}
\begin{eqnarray*}
P\{X_0(R^3(t))\in dx|N(t)=2\}&=&P\{X_1(R^3(t))\in dx|N(t)=0\}\\
&=&P\{B_1(c^2\Gamma(t))\in dx|B(t)>0\}\\
&=&\frac{dx 2}{\pi t}\int_0^t\frac{\sqrt{s}}{\sqrt{(t-s)}}\frac{e^{-\frac{x^2}{2c^2 s}}}{\sqrt{2\pi s}c}ds,
\end{eqnarray*}
\begin{eqnarray*}
P\{X_0(|B_1(t)|)\in dx|N(t)=2\}&=&P\{X_1(|B_1(t)|)\in dx|N(t)=0\}\\
&=&P\{B_1(c^2\Gamma(t))\in dx|B(t)<0\}\\
&=&\frac{dx 2}{\pi t}\int_0^t\frac{\sqrt{t-s}}{\sqrt{s}}\frac{e^{-\frac{x^2}{2c^2 s}}}{\sqrt{2\pi s}c}ds,
\end{eqnarray*}
where $B_1(t)$ is a independent Brownian motion with respect to $B(t)$.
\end{remark}

\subsection{Compositions with Gamma random times}

In this part of the paper, we deal with a second class of random times different with respect to the previous one. 
 We indicate with $G_\alpha(t)$ a Gamma random process, with parameter $\alpha>0$, governed by the density law $g_\alpha(s,t)=\frac{t^\alpha}{\Gamma(\alpha)}s^{\alpha-1}e^{-ts}, s>0$. Analogously to the Bessel case, first of all we study the conditional probability distribution.
 
 \begin{theorem} \label{teogamma}
 Given $N(t)=n\geq 1$, such that $v_m>\frac{\alpha }{2}-1,$ the random process $X_m(G_\alpha(t))$ has the following conditional probabilities
 \begin{eqnarray}\label{distgamma}
 &&P\left\{X_m(G_\alpha(t))\in dx|N(t)=n\right\}\\
 &&=\frac{dx}{\Gamma(\frac{\alpha+1}{2})B(\frac\alpha 2,v_m-\frac\alpha 2+1)} \int_0^1dw w^{\frac{\alpha}{2}-1}(1-w)^{v_m-\frac\alpha 2}\frac{t}{\sqrt{\pi w}c}\left(\frac{t|x|}{2c\sqrt{w}}\right)^{\frac\alpha 2}K_{-\frac\alpha 2}\left(\frac{t|x|}{c\sqrt{w}}\right)\notag \end{eqnarray}
where $K_\mu(x)$ is the second type modified Bessel function.
 \end{theorem}
 \begin{proof}
 The random process $X_m(G_\alpha(t))$ admits as support the real line, hence we get that
 \begin{eqnarray*} 
&&P\left\{X_m(G_\alpha(t))\in dx|N(t)=n\right\}\\
&&=\int_0^\infty P\left\{X_m(G_\alpha(t))\in dx|N(t)=n, G_\alpha(t)=s\right\}P\{G_\alpha(t)\in ds\}\\
&&=\frac{dx}{2\pi}\frac{\Gamma(v_m+1)\Gamma(v_m)}{\Gamma(2v_m)}\frac{t^{\alpha }}{\Gamma(\alpha)}\int_0^\infty\left(\frac{2}{cs}\right)^{v_m}(c^2s^2-x^2)^{v_m-\frac12}{\bf 1}_{\{|x|<cs\}}s^{\alpha -1}e^{-ts}ds
\end{eqnarray*}
Then
\begin{eqnarray*}
&&E\left\{e^{i\alpha X_m(G_\alpha(t))}|N(t)=n\right\}\\
&&=\Gamma(v_m+1)\left(\frac{2}{\alpha c}\right)^{v_m}\frac{t^{\alpha }}{\Gamma(\alpha)}\int_0^\infty \frac{J_{v_m}(\alpha cs)}{s^{v_m}}s^{\alpha -1}e^{-ts}ds\\
&&=\Gamma(v_m+1)\frac{t^{\alpha }}{\Gamma(\alpha)}\sum_{k=0}^\infty \frac{(-1)^k}{k!\Gamma(k+v_m+1)}\left(\frac{\alpha c}{2}\right)^{2k}\int_0^\infty s^{2k+\alpha -1}e^{-ts}ds\\
&&= \frac{\Gamma(v_m+1)}{\Gamma(\alpha)}\sum_{k=0}^\infty \frac{(-1)^k\Gamma(2k+\alpha)}{k!\Gamma(k+v_m+1)}\left(\frac{\alpha c }{2t}\right)^{2k}\\
&&=(\text{by duplication formula})\\
&&= \frac{\Gamma(v_m+1)}{\sqrt{\pi}\Gamma(\alpha)}\sum_{k=0}^\infty \frac{(-1)^k\Gamma(k+\frac \alpha2)\Gamma(k+\frac{\alpha +1}{2})2^{2k+\alpha-1}}{k!\Gamma(k+v_m+1)}\left(\frac{\alpha c }{2t}\right)^{2k}\\
&&= \frac{\Gamma(v_m+1)2^{\alpha -1}}{\sqrt{\pi}\Gamma(\alpha)\Gamma(v_m-\frac\alpha2+1)}\sum_{k=0}^\infty \frac{(-1)^k}{k!}\left(\frac{\alpha c}{t}\right)^{2k}\Gamma\left(k+\frac{\alpha+1}{2}\right)B\left(k+\frac\alpha2,v_m-\frac\alpha2+1\right)\\
&&= \frac{\Gamma(v_m+1)}{\Gamma(\frac\alpha2)\Gamma(\frac{\alpha +1}{2})\Gamma(v_m-\frac\alpha2+1)}\sum_{k=0}^\infty \frac{(-1)^k}{k!}\left(\frac{\alpha c}{t}\right)^{2k}\Gamma\left(k+\frac{\alpha+1}{2}\right)B\left(k+\frac\alpha2,v_m-\frac\alpha2+1\right)\\
&&=\frac{1}{\Gamma(\frac{\alpha +1}{2}) B(\frac\alpha2,v_m-\frac\alpha2+1)}\sum_{k=0}^\infty  \frac{(-1)^k}{k!}\left(\frac{\alpha c}{t}\right)^{2k}\int_0^\infty e^{-z}z^{k+\frac{\alpha+1}{2}-1}dz\int_0^1w^{k+\frac\alpha2-1}(1-w)^{v_m-\frac\alpha2}dw\\
&&=\frac{1}{\Gamma(\frac{\alpha +1}{2}) B(\frac\alpha2,v_m-\frac\alpha2+1)}\int_0^\infty e^{-z}z^{\frac{\alpha+1}{2}-1}dz \int_0^1w^{\frac\alpha2-1}(1-w)^{v_m-\frac\alpha2}dwe^{-\left(\frac{\alpha c}{t}\right)^2zw}.
\end{eqnarray*}
By inverting the above quantity, we obtain that
 \begin{eqnarray}\label{densproof}
 &&\frac{1}{dx}P\left\{X_m(G_\alpha(t))\in dx|N(t)=n\right\}\\
 &&=\frac{1}{\Gamma(\frac{\alpha+1}{2})}\int_0^\infty dz e^{-z}z^{\frac{\alpha +1}{2}-1}\frac{1}{B(\frac\alpha 2,v_m-\frac\alpha 2+1)} \int_0^1dw w^{\frac{\alpha}{2}-1}(1-w)^{v_m-\frac\alpha 2}\frac{te^{-\frac{t^2x^2}{4c^2zw}}}{ \sqrt{4\pi zw}c}\notag\\
 &&=\frac{1}{\Gamma(\frac{\alpha+1}{2})}\frac{1}{B(\frac\alpha 2,v_m-\frac\alpha 2+1)} \int_0^1dw w^{\frac{\alpha}{2}-1}(1-w)^{v_m-\frac\alpha 2}\frac{t}{\sqrt{\pi w}c}\left(\frac{t|x|}{2c\sqrt{w}}\right)^{\frac\alpha 2}K_{-\frac\alpha 2}\left(\frac{t|x|}{c\sqrt{w}}\right)\notag
 \end{eqnarray}
 where in the last step we have used the following integral representation
 $$K_\mu(x)=\frac12\left(\frac x2\right)^\mu\int_0^\infty e^{-z-\frac{x^2}{4z}}z^{-\mu-1}dz.$$
 Moreover, from \eqref{densproof} it is easy to verify that
 $$\int_{-\infty}^\infty P\left\{X_m(G_\alpha(t))\in dx|N(t)=n\right\}=1.$$

\end{proof}
\begin{remark}
From \eqref{densproof} it is straightforward to observe that
 $$P\left\{X_m(G_\alpha(t))\in dx|N(t)=n\right\}$$
is equal to the density of a Gaussian random variable with variance $\frac{2c^2}{t^2}ZW$ where $Z$ is gamma variable with parameter $\frac{\alpha+1}{2}$ and 1, while $W\sim B(\frac\alpha 2,v_m-\frac\alpha 2 +1)$.
\end{remark}
\begin{remark}
analogously to the Bessel case, the approach developed in the previous proof is enought to show that 
\begin{align}\label{distgamma0}
&P\left\{X_\nu(G_\alpha(t))\in dx|N(t)=0\right\}\\
&=\frac{dx}{\Gamma(\frac{\alpha+1}{2})B(\frac\alpha 2,\nu-\frac\alpha 2+1)} \int_0^1dw w^{\frac{\alpha}{2}-1}(1-w)^{\nu-\frac\alpha 2}\frac{t}{\sqrt{\pi w}c}\left(\frac{t|x|}{2c\sqrt{w}}\right)^{\frac\alpha 2}K_{-\frac\alpha 2}\left(\frac{t|x|}{c\sqrt{w}}\right)\notag
 \end{align}
 with $\nu>\frac \alpha 2-1$.
 \end{remark}

For $\alpha=1$, $G_1(t)$ becomes an exponential random process. The exponential clock permits us to derive an interesting interpretation of the probability \eqref{distgamma}. Indeed, since
$$K_{\pm \frac12}(x)=\sqrt{\frac{\pi}{2x}}e^{-x}$$
we get that
 \begin{eqnarray}\label{distlap}
P\left\{X_m(G_1(t))\in dx|N(t)=n\right\}=\frac{1}{B(\frac1 2,v_m+\frac1 2)} \int_0^1dw w^{-\frac{1}{2}}(1-w)^{v_m-\frac12}\frac{t}{2c\sqrt{ w}}e^{-\frac{t|x|}{c\sqrt{w}}}
 \end{eqnarray}
 for $n\geq 0$, being the condition $v_m>-\frac 12$ always satisfied. It means that the conditioned probability of $X_m(G_1(t))$ is equivalent to the distribution of a Laplace random variable with parameter $\frac{t}{2c\sqrt{ W}}$, where $W\sim B(\frac12,v_m+\frac12)$. 
 \begin{remark}
 The process $X_0(G_1(t))$ and the occupation time $\Gamma(t)$ are linked by the following relationships
 \begin{align*}
P\left\{X_0(G_1(t))\in dx|N(t)=0\right\}&=\frac{dx}{\pi}\int_0^tP\{\Gamma(t)\in ds\}\frac{1}{2cs}e^{-\frac{|x|}{cs}}ds\\\\
 P\left\{X_0(G_1(t))\in dx|N(t)=2\right\}&=P\left\{X_1(G_1(t))\in dx|N(t)=0\right\}\\
&= \frac{dx 2}{\pi t}\int_0^tP\{\Gamma(t)\in ds|B(t)<0\}\frac{1}{2cs}e^{-\frac{|x|}{cs}}ds
 \end{align*}
 as it is possible to derive by \eqref{distgamma0} and \eqref{distlap}.
 \end{remark}
 
 Now, by means of the same steps used in the proof of Theorem \ref{cor}, it in not hard to prove the result contained in the next Theorem. Therefore, we omitted the proof.

 \begin{theorem}\label{cor2}
 For $\alpha=1$, we have that
 \begin{equation}
 P\left\{X_m(G_1(t))\in dx\right\}=dx\int_0^{ct}\frac{t^2}{2y}e^{-\frac{t^2|x|}{y}}q^m(y,t)dy
 \end{equation}
 where $q^m(y,t)=2p^m(y,t),\,m=0,1$.
 \end{theorem}
Theorem \ref{cor2} claims that $X_m(t)$ with exponential time, is equivalent, in distribution, to a Laplace random variable with parameter $\frac{t^2}{|X_m(t)|}$.
 
 As for the Bessel random time, also in this case we consider a random time obtained iterating $l$ times a Gamma random variable. Hence, we indicate the iterated Gamma process with
\begin{equation*}
\mathcal{G}_l(t)=G_{\alpha_1}(G_{\alpha_2}(\ldots (G_{\alpha_{l+1}}(t))\ldots))
\end{equation*} 
having density law given by
$$g_l(s,t)=\int_0^\infty\frac{t_1^{\alpha_1}}{\Gamma(\alpha_1)}s^{\alpha_1-1}e^{-t_1s}dt_1\int_0^\infty\frac{t_2^{\alpha_2}}{\Gamma(\alpha_2)}t_1^{\alpha_2-1}e^{-t_2t_1}dt_2\ldots \int_0^\infty\frac{t^{\alpha_{l+1}}}{\Gamma(\alpha_{l+1})}t_l^{\alpha_{l+1}-1}e^{-t_lt}dt_{l}$$
where $\alpha_{j+1},j=0,1,...,n,$ is strictly positive. Therefore, by using the same approach contained in the proof of Theorem \ref{teogamma}, the following result yields.
 \begin{theorem}
 Given $N(t)=n$, such that $v_m>\frac {\alpha_1}{2}-1$, we have that
 \begin{eqnarray}
&& P\left\{X_m(\mathcal{G}_l(t))\in dx|N(t)=n\right\}\\
 &&=\frac{dx}{\Gamma(\frac{\alpha_1+1}{2})B(\frac {\alpha_1}{2},v_m-\frac {\alpha_1}{2}+1)}
\int_0^\infty\frac{t_2^{\alpha_2}}{\Gamma(\alpha_2)}t_1^{\alpha_2-1}e^{-t_2t_1}dt_1\ldots \int_0^\infty\frac{t^{\alpha_{l+1}}}{\Gamma(\alpha_{l+1})}t_l^{\alpha_{l+1}-1}e^{-t_lt}dt_{l}\notag\\
 &&\quad \times\int_0^1w^{\frac{\alpha_1}{2}-1}(1-w)^{v_m-\frac{\alpha_1} {2}}\frac{t_1}{\sqrt{\pi w}c}\left(\frac{t_1|x|}{2c\sqrt{w}}\right)^{\frac{\alpha_1} {2}}K_{-\frac{\alpha_1}{ 2}}\left(\frac{t_1|x|}{c\sqrt{w}}\right)dw\notag
 \end{eqnarray}
 with $x\in \mathbb{R}$.
 \end{theorem}
Furthermore, at this point also the probability $P \left\{X_\nu(\mathcal{G}_l(t))\in dx|N(t)=0\right\}$ immediately follows.

By setting $\alpha_1=1$, it is not an hard task to provide the following unconditonal distribution 
   \begin{eqnarray}
 &&P\left\{X_m(G_{1}(G_{\alpha_2}(\cdots(G_{\alpha_{l+1}}(t))\cdots)))\in dx\right\}\\
 &&=dx
 \int_0^\infty\frac{t_2^{\alpha_2}}{\Gamma(\alpha_2)}t_1^{\alpha_2-1}e^{-t_2t_1}dt_1\ldots \int_0^\infty\frac{t^{\alpha_{l+1}}}{\Gamma(\alpha_{l+1})}t_l^{\alpha_{l+1}-1}e^{-t_lt}dt_{l}\int_0^{ct}\frac{t^2}{2t_1y}e^{-\frac{t^2|x|}{t_1y}}q^m(y,t)dy.\notag
 \end{eqnarray}
 Hence, $X_m(G_{1}(G_{\alpha_2}(\cdots(G_{\alpha_{l+1}}(t))\cdots)))$ is distributed as a Laplace random variable with parameter given by
 $$\frac{t^2}{G_{\alpha_1}(\cdots(G_{\alpha_{l}}(t))\cdots)|X_m(t)|}.$$

In order to conclude the discussion on the random times, we note that $R^d(t)$ and $G_\alpha(t)$ can be mixed obtaining a new class of random times. Obviously, $R^d(t)$ and $G_\alpha(t)$ are thought to be mutually independent. Now, we present the following Theorem. 
\begin{theorem}\label{teomix}
For $n\geq 1$ and $v_m>\frac d2-1$, we have that
\begin{eqnarray}
&&P\left\{X_m(R^d(G_\alpha(t)))\in dx|N(t)=n\right\}\\
&&=\frac{dx}{\Gamma(\alpha)B\left(\frac d2,v_m-\frac d2+1\right)}\int_0^1dw w^{\frac d2-1}(1-w)^{v_m-\frac d2}\frac{\sqrt{2t}}{\sqrt{\pi w}c}\left(\frac{\sqrt{t}|x|}{c\sqrt{2w}}\right)^{\alpha-\frac12}K_{-\alpha+\frac12}\left(\frac{\sqrt{2t}|x|}{c\sqrt{w}}\right)\notag
\end{eqnarray}
\end{theorem}
\begin{proof} Since $R^d(G_\alpha(t))$ has density given by
$$ \frac{t^\alpha}{\Gamma(\alpha)\Gamma(\frac d2)2^{\frac d2-1}}\int_0^\infty\frac{r^{d-1}}{z^{\frac d2}}e^{-\frac{r^2}{2z}}z^{\alpha-1}e^{-t z}dz,$$
we can write that
\begin{eqnarray*}
&&P\left\{X_m(R^d(G_\alpha(t)))\in dx|N(t)=n\right\}\\
&&=\int_0^\infty P\left\{X_m(R^d(G_\alpha(t)))\in dx|N(t)=n,R^d(G_\alpha(t))=s \right\}P\{R^d(G_\alpha(t))\in  ds\}\\
&&=\frac{\Gamma(v_m+1)\Gamma(v_m)}{2\pi\Gamma(2v_m)}\frac{t^\alpha dx}{\Gamma(\alpha)\Gamma(\frac d2)2^{\frac d2-1}}\\
&&\quad\times\int_0^\infty ds\int_0^\infty dz\left(\frac{2}{c s}\right)^{v_m}(c^2s^2-x^2)^{v_m-\frac12}{\bf 1}_{\{|x|<cs\}}\frac{s^{d-1}}{z^{\frac d2}}e^{-\frac{s^2}{2z}}z^{\alpha-1}e^{-t z}.
\end{eqnarray*}
Therefore, by means of the Corollary \ref{corcf},  we are able to explicit the Fourier transform of the previous probability distribution as follows 
\begin{eqnarray*}
&&\int_{-\infty}^{+\infty}e^{i\alpha x}P\left\{X_m(R^d(G_\alpha(t)))\in dx|N(t)=n\right\}\\
&&=\Gamma(v_m+1)\left(\frac{2}{\alpha c}\right)^{v_m}\frac{t^\alpha}{\Gamma(\alpha)\Gamma(\frac d2)2^{\frac d2-1}}\int_0^\infty z^{\alpha-\frac d2-1}e^{-t z}dz \int_0^\infty J_{v_m}(\alpha cs)s^{d-v_m-1}e^{-\frac{s^2}{2z}}ds\\
&&=\frac{t^\alpha\Gamma(v_m+1)}{\Gamma(\alpha)\Gamma(\frac d2)2^{\frac d2-1}}\int_0^\infty z^{\alpha-\frac d2-1}e^{-t z}dz\sum_{k=0}^\infty\frac{(-1)^k}{k!\Gamma(k+v_m+1)}\left(\frac{\alpha c}{2}\right)^{2k}\int_0^\infty s^{2k+d-1} e^{-\frac{s^2}{2z}}ds\\
&&=\left(y=\frac{s^2}{2z}\right)\\
&&=\frac{t^\alpha\Gamma(v_m+1)}{\Gamma(\alpha)\Gamma(\frac d2)}\int_0^\infty z^{\alpha-1}e^{-t z}dz\sum_{k=0}^\infty\frac{(-1)^k}{k!\Gamma(k+v_m+1)}\left(\frac{\alpha^2 c^2 z}{2}\right)^{k}\int_0^\infty y^{k+\frac d2-1} e^{- y}dy\\
&&=\frac{t^\alpha\Gamma(v_m+1)}{\Gamma(\alpha)\Gamma(\frac d2)}\int_0^\infty z^{\alpha-1}e^{-t z}dz\sum_{k=0}^\infty\frac{(-1)^k\Gamma(k+\frac d2)}{k!\Gamma(k+v_m+1)}\left(\frac{\alpha^2 c^2 z}{2}\right)^{k}\\
&&=\frac{t^\alpha\Gamma(v_m+1)}{\Gamma(\alpha)\Gamma(\frac d2)\Gamma\left(v_m-\frac d2+1\right)}\int_0^\infty z^{\alpha-1}e^{-t z}dz\sum_{k=0}^\infty\frac{(-1)^k}{k!}B\left(k+\frac d2,v_m-\frac d2+1\right)\left(\frac{\alpha^2 c^2 z}{2}\right)^{k}\\
&&=\frac{t^\alpha}{\Gamma(\alpha)B\left(\frac d2,,v_m-\frac d2+1\right)}\int_0^\infty z^{\alpha-1}e^{-t z}dz\sum_{k=0}^\infty\frac{(-1)^k}{k!}\left(\frac{\alpha^2 c^2 z}{2}\right)^{k}\int_0^1w^{k+\frac d2-1}(1-w)^{v_m-\frac d2}dw\\
&&=\frac{t^\alpha}{\Gamma(\alpha)B\left(\frac d2,,v_m-\frac d2+1\right)}\int_0^\infty z^{\alpha-1}e^{-t z}dz\int_0^1w^{\frac d2-1}(1-w)^{v_m-\frac d2}e^{-\frac{\alpha^2 c^2zw}{2}}dw
\end{eqnarray*}
By inverting the above characteristic function, we obtain that
\begin{eqnarray*}
&&P\left\{X_m(R^d(G_\alpha(t)))\in dx|N(t)=n\right\}\\
&&=\frac{t^\alpha}{\Gamma(\alpha)B\left(\frac d2,,v_m-\frac d2+1\right)}\int_0^1w^{\frac d2-1}(1-w)^{v_m-\frac d2}dw \int_0^\infty z^{\alpha-1}\frac{e^{-tz-\frac{x^2}{2c^2zw}}}{\sqrt{2\pi zw}c}dz\\
&&=(u=t z)\\
&&=\frac{1}{\Gamma(\alpha)B\left(\frac d2,,v_m-\frac d2+1\right)}\int_0^1w^{\frac d2-1}(1-w)^{v_m-\frac d2}dw\int_0^\infty u^{\alpha-\frac 12-1}\frac{\sqrt{t}e^{-u-\frac{x^2t}{2c^2uw}}}{\sqrt{2\pi w}c}du\\
&&=\frac{1}{\Gamma(\alpha)B\left(\frac d2,,v_m-\frac d2+1\right)}\int_0^1dw w^{\frac d2-1}(1-w)^{v_m-\frac d2}\frac{\sqrt{2t}}{\sqrt{\pi w}c}\left(\frac{\sqrt{t}|x|}{c\sqrt{2w}}\right)^{\alpha-\frac12}K_{-\alpha+\frac12}\left(\frac{\sqrt{2t}|x|}{c\sqrt{w}}\right).
\end{eqnarray*}
\end{proof}
\begin{remark}
For $\alpha=1$, one has
$$P\{X_m(R^d(G_1(t)))\in dx|N(t)=n\}=\frac{dx}{B\left(\frac d2,,v_m-\frac d2+1\right)}\int_0^1dw w^{\frac d2-1}(1-w)^{v_m-\frac d2}\frac{\sqrt{2t}}{2c\sqrt{w}}e^{-\frac{\sqrt{2t}|x|}{c\sqrt{w}}}$$
and for $d=1$, after some calculations, the following distribution yields 
$$P\{X_m(|B(G_1(t))|)\in dx\}=dx\int_0^{ct}\frac{\sqrt{2}t^{\frac32}}{2y}e^{-\frac{\sqrt{2}t^{\frac32}|x|}{y}}q^m(y,t)dy.$$
\end{remark}
\begin{remark}
Let $X$ and $Y$be two independent random variables distributed as a Gaussian with mean zero and variance respectively equal to $\sigma^2$ and 1. It is well-known that $V=XY$ has the following density law 
\begin{eqnarray*}
f_V(v)=\int_{-\infty}^{+\infty}\frac{e^{-\frac{x^2}{2\sigma^2}}}{\sqrt{2\pi}\sigma}\frac{e^{-\frac{v^2}{2x^2}}}{\sqrt{2\pi}}dx=\frac{1}{\pi\sigma}K_0\left(\frac{|v|}{\sigma}\right).
\end{eqnarray*}
These considerations and Theorem \ref{teomix} permit us to state that for $\alpha=\frac12$ and $d=1$,
$$P\left\{X_m(|B(G_\frac12(t))|)\in dx|N(t)=n\right\}=\frac{dx}{B(\frac 1 2,\frac {n+1}{2})} \int_0^1dw w^{\frac{d}{2}-1}(1-w)^{v_m-\frac d 2}\frac{\sqrt{2t}}{\pi c\sqrt{ w}}K_{0}\left(\frac{\sqrt{2t}|x|}{c\sqrt{w}}\right)$$
holds for each $n$. In other words,
conditionally on the number of Poisson events during the interval $[0,t]$, the law of $X_m(|B(G_{\frac12}(t))|)$ is equivalent to the distribution of the product of a standard Gaussian and a Normal random variable, indipendent from previous one, with mean zero and variance $\frac{c\sqrt{W}}{\sqrt{2t}}$, with $W\sim B(\frac 12,\frac {n+ 1}{2})$.

\end{remark}

\section{Some remarks on random motions in higher spaces}
Let $\underline{{\bf X}}_2(t)=(X_1(t),X_2(t))$ and $\underline{{\bf X}}_4(t)=(X_1(t),X_2(t),X_3(t),X_4(t))$ be respectively a planar and four-dimensional random flight, then the results presented in the Section \ref{sec:comp}
can be extended to these random processes. 
 
We also use the follwing notations: $\underline{{\bf \alpha}}_2=(\alpha_1,\alpha_2)$, $\underline{{\bf \alpha}}_4=(\alpha_1,\alpha_2,\alpha_3,\alpha_4)$, $\underline{{\bf x}}_2=(x_1,x_2)$, $\underline{{\bf x}_4}=(x_1,x_2,x_3,x_4)$. Futhermore, let $||\cdot||$ be the euclidean norm and $<\cdot,\cdot>$ the scalar product. As proved in Orsingher and De Gregorio (2007b) the characteristic function and the conditional probabilities of the random flights in the plane are given by
\begin{align*}
&E\left\{e^{i<\underline{\alpha}_2, {\bf X}_2(t)>}|N(t)=n\right\}=\frac{\Gamma\left(\frac n2+1\right)2^{\frac n2}}{( ct||\underline{{\bf   \alpha}}_2||)^{\frac n2}}J_{\frac n2}(  ct||\underline{{\bf \alpha}}_2||),
\\
&p_n(||\underline{{\bf x}}_2||,t)=
\frac{n}{2\pi (ct)^n}(c^2t^2-||\underline{{\bf x}}_2||^2)^{\frac n2-1},
\end{align*}  
while for $\underline{{\bf X}}_4(t)$, one has
\begin{align*}
&E\left\{e^{i<\underline{\alpha}_4, {\bf X}_4(t)>}|N(t)=n\right\}=\frac{\Gamma\left(n+2\right)2^{ n+1}}{( ct||\underline{{\bf \alpha}}_4||)^{n+1}}J_{n+1}(  ct||\underline{{\bf \alpha}}_4||),\\
& p_n(||\underline{{\bf x}}_4||,t)=\frac{n(n+1)}{\pi^2(ct)^{2n+2}} 
(c^2t^2-||\underline{{\bf x}}_4||^2)^{n-1}.
\end{align*}  
Therefore, for the planar random flights at Bessel random time, namely $(X_1(R^d(t)),X_2(R^d(t)))$, by following the same steps of the proof of Theorem \ref{stopdist}, we have that
\begin{align}
&P\left\{X_1(R^d(t))\in dx_1,X_2(R^d(t))\in dx_2|N(t)=n\right\}\label{condprobpr}\\
&=\frac{dx_1dx_2}{B\left(\frac d2,\frac n2-\frac d2+1\right)}\int_0^1w^{\frac d2-1}(1-w)^{\frac n2-\frac d2}\frac{e^{-\frac{||\underline{{\bf x}}_2||^2}{2c^2t w}}}{2\pi twc^2}dw\notag
\\ \notag\\
&P\left\{X_1(G_\alpha(t))\in dx_1,X_2(G_\alpha(t))\in dx_2|N(t)=n\right\}\\
&=\frac{dx_1dx_2}{\Gamma(\frac{\alpha+1}{2})B(\frac\alpha 2,v_m-\frac\alpha 2+1)} \int_0^1dw w^{\frac{\alpha}{2}-1}(1-w)^{v_m-\frac\alpha 2}\frac{t}{\sqrt{\pi w}c}\left(\frac{t||\underline{{\bf x}}_2||}{2c\sqrt{w}}\right)^{\frac\alpha 2}K_{-\frac\alpha 2}\left(\frac{t||\underline{{\bf x}}_2||}{c\sqrt{w}}\right)\notag 
\end{align}
which hold if $v_m>\frac d2-1$ again.
We observe that from the probability \eqref{condprobpr} for $d=1$, we reobtain the result $(3.3)$ showed in Beghin and Orsingher (2009). Similar considerations yield for the four-dimensional random flights.

The random flights in higher spaces have directions uniformly distributed on a multidimensional hypersphere. It would be interesting to consider a model, for example in the plane, with a different density law with respect to the uniform one. Let us consider a planar random flight with density law similar to \eqref{dens:vel}, for example
$$f(\theta)=\frac{1}{2\pi}\sin^2\theta,\quad \theta\in[0,2\pi].$$
In this case, we obtain a random process describing a motion tending to move in a land of the plane with high probability. Therefore, we obtain a random motion with drift, which is persistent along a specific portion of the plane. Hence, in order to calculate the characteristic function, we need of the following integral
$$\int_0^{2\pi}\exp\{iz(\alpha\cos\theta+\beta\sin\theta)\}\sin^{2}\theta d\theta$$
which it is work out as follows

\begin{eqnarray*}
&&\int_0^{2\pi}\exp\{iz(\alpha\cos\theta+\beta\sin\theta)\}\sin^{2}\theta d\theta\\
&&=\sum_{k=0}^\infty\frac{(iz)^k}{k!}\int_0^{2\pi}(\alpha\cos\theta+\beta\sin\theta)^k\sin^{2}\theta d\theta\\
&&=\sum_{k=0}^\infty\frac{(iz)^k}{k!}\sum_{r=0}^k\binom{k}{r}\alpha^r\beta^{k-r}\int_0^{2\pi}\cos^{k}\theta\sin^{k-r+2}\theta d\theta.
\end{eqnarray*}
Now, the last integral has to be splitted in two parts: $\int_0^{2\pi}=\int_0^{\pi }+\int_{\pi}^{2\pi }$. Hence, by performing a change of variable $\theta'=\theta-\pi$ in the second integral, we observe by \eqref{wallis} that the previous sum is equal to zero if $k$ and $r$ are odd. By splitting the first integral on $(0,\frac\pi 2)$ and after a change of variable analogous to the previous one, we can write that
\begin{eqnarray*}
&&\int_0^{2\pi}\exp\{iz(\alpha\cos\theta+\beta\sin\theta)\}\sin^{2}\theta d\theta\\
&&=4\sum_{k=0}^\infty\frac{(iz)^{2k}}{(2k)!}\sum_{r=0}^k\binom{2k}{2r}\alpha^{2r}\beta^{2(k-r)}\int_0^{\pi/2}\cos^{2r}\theta\sin^{2(k-r+1)}\theta d\theta\\
&&=2\sum_{k=0}^\infty(-1)^k\frac{z^{2k}}{(2k)!}\sum_{r=0}^k\binom{2k}{2r}\alpha^{2r}\beta^{2(k-r)}\frac{\Gamma(r+\frac12)\Gamma(k-r+1+\frac12)}{\Gamma(k+2)}\\
&&=2\pi\sum_{k=0}^\infty(-1)^k\frac{z^{2k}}{(2k)!}\sum_{r=0}^k\binom{2k}{2r}\alpha^{2r}\beta^{2(k-r)}\frac{2^{-2k}\Gamma(2r)\Gamma(2(k-r+1))}{\Gamma(k+2)\Gamma(r)\Gamma(k-r+1)}\\
&&=2\pi\sum_{k=0}^\infty(-1)^k\frac{z^{2k}}{\Gamma(k+2)2^{2k}}\sum_{r=0}^k\frac{2(k-r)+1}{2r\Gamma(r)\Gamma(k-r+1)}\alpha^{2r}\beta^{2(k-r)}\\
&&=2\pi\sum_{k=0}^\infty(-1)^k\frac{z^{2k}}{(k+1)!2^{2k}}\left\{\sum_{r=0}^k\frac{\alpha^{2r}\beta^{2(k-r)}}{2(r!)(k-r)!}+\sum_{r=0}^k\frac{\alpha^{2r}\beta^{2(k-r)}}{r!(k-r-1)!}\right\}\\
&&=2\pi\sum_{k=0}^\infty(-1)^k\frac{z^{2k}}{(k+1)!2^{2k}}\left\{\frac{1}{2(k!)}\sum_{r=0}^k\binom{k}{r}\alpha^{2r}\beta^{2(k-r)}+\frac{\beta^2}{(k-1)!}\sum_{r=0}^{k-1}\binom{k-1}{r}\alpha^{2r}\beta^{2(k-1-r)}\right\}\\
&&=2\pi\left\{\sum_{k=0}^\infty\frac{(-1)^kz^{2k}}{2(k+1)!k!2^{2k}}(\sqrt{\alpha^2+\beta^2})^{2k}+\beta^2\sum_{k=1}^\infty\frac{(-1)^kz^{2k}}{(k+1)!(k-1)!2^{2k}}(\sqrt{\alpha^2+\beta^2})^{2(k-1)}\right\}\\
&&=2\pi\left\{\frac{J_1(z\sqrt{\alpha^2+\beta^2})}{z\sqrt{\alpha^2+\beta^2}}-\beta^2\sum_{l=0}^\infty\frac{(-1)^l}{l!(l+2)!}\left(\frac z2\right)^{2l+2}(\sqrt{\alpha^2+\beta^2})^{2l}\right\}\\
&&=2\pi\left\{\frac{J_1(z\sqrt{\alpha^2+\beta^2})}{z\sqrt{\alpha^2+\beta^2}}-\frac{\beta^2}{\alpha^2+\beta^2}J_2(z\sqrt{\alpha^2+\beta^2})\right\}.
\end{eqnarray*}
Then, the characteristic function becomes 
\begin{eqnarray}\label{planarcf}
&&E\left\{e^{i<\underline{\alpha}_2,\underline{ {\bf X}}_2(t)>}|N(t)=n\right\}\\
&&=\frac{n!}{t^n}\int_0^tds_1\cdots\int_{s_{n-1}}^tds_{n}\prod_{j=1}^{n+1}\left\{\frac{J_1(c(s_j-s_{j-1})||\underline{\alpha}_2||)}{c(s_j-s_{j-1})||\underline{\alpha}_2||}-\frac{\alpha_2^2}{||\underline{\alpha}_2||^2}J_2(c(s_j-s_{j-1})||\underline{ \alpha}_2||)\right\}.\notag
\end{eqnarray}
From \eqref{planarcf} emerges that as expected an asimmetry is introduced by $f(\theta)$, because the particle will tend to maintains the same direction. Moreover, the inversion of the characteristic function is quite difficult, therefore seem to be not possible to obtain the explicit probability distribution of $\underline{{\bf X}}_2$ at time $t$ by means of this approach.

 \end{document}